\documentclass[12pt, twoside, leqno]{article}

\usepackage{amsmath,amsthm}
\usepackage{amssymb}
\usepackage{enumerate}

\usepackage{enumitem}

\usepackage{mathtools}
\usepackage{ascmac}
\usepackage[all]{xy}
\usepackage{yhmath}
\usepackage{geometry}

\newtheorem{theorem}{Theorem}[section]
\newtheorem{corollary}[theorem]{Corollary}
\newtheorem{lemma}[theorem]{Lemma}
\newtheorem{proposition}[theorem]{Proposition}

\theoremstyle{definition}
\newtheorem{definition}[theorem]{Definition}
\newtheorem{remark}[theorem]{Remark}
\newtheorem{example}[theorem]{Example}

\numberwithin{equation}{section}

\makeatletter
\renewcommand{\p@enumii}{}
\makeatother

\newcommand{\relmid}{\mathrel{}\middle|\mathrel{}}

\newcommand{\Gal}{\operatorname{Gal}}
\newcommand{\coef}{\operatorname{coef}}

\begin{document}

\title{Ramification groups of Galois extensions over local fields of positive characteristic with Galois group isomorphic to the group of unitriangular matrices}


\author{Koto Imai\\
Graduate School of Mathematical Sciences\\ 
The University of Tokyo\\
153-8914, Tokyo, Japan\\
E-mail: imai.koto.541@gmail.com}

\date{}

\maketitle


\renewcommand{\thefootnote}{}



\renewcommand{\thefootnote}{\arabic{footnote}}
\setcounter{footnote}{0}


\begin{abstract}
	We study the ramification groups of finite Galois extensions $L/K$ of a complete discrete valuation field $K$ of equal characteristic $p>0$ with perfect residue field and Galois group isomorphic to the group of unitriangular matrices $UT_n(\mathbb{F}_p)$ over $\mathbb{F}_p$. We show that the upper ramification breaks can be expressed as a linear function of the valuation of the entries of a matrix directly constructed from the coefficients of a defining equation of the extension. This allows us to compute the ramification groups without using any elements of $L-K$.  
\end{abstract}

\section{Introduction}\label{sc001}
Let $ K $ be a complete discrete valuation field of equal characteristic $ p > 0 $ with perfect residue field $ k $. For any finite Galois extension of $ K $, one can define the upper and lower ramification groups, or filtrations. When the extension is abelian, the jumps in the upper ramification filtration can be described in terms of the coefficients of a defining equation of the extension.
For example, if $L/K$ is an Artin-Schreier extension defined by $x^p+a=x$, where $a\in K$ satisfies $p\nmid v_{K}(a)<0$, then the unique upper ramification break of $L/K$ is given by $-v_{K}(a)$. Kanesaka and Sekiguchi \cite{KS} generalized this result to finite cyclic cases, using Witt vectors. Since every finite abelian group can be decomposed into a direct product of cyclic groups, this allows one to compute the upper ramification breaks of any finite abelian extension $L/K$ in terms of the coefficients of a defining equation of the extension.

In the abelian case, the Hasse-Arf theorem asserts that the upper ramification breaks are integers. On the other hand, in the non-abelian setting, the situation is more complicated, and upper ramification breaks can be a fraction with a denominator of a divisor of the order of the Galois group. Specifically, \cite{Im} calculates the upper ramification breaks for some finite Galois extensions of maximal nilpotency class and shows that in some cases, they belong to $\frac{1}{p}\mathbb{Z}-\mathbb{Z}$. Furthermore, \cite{El} shows that if the Galois group is non-abelian and of order $p^3$, then the extension can have non-integer upper ramification breaks, except for the case where the Galois group is isomorphic to the dihedral group $D_8$ of order 8.

In this paper, we study the ramification groups of totally wildly ramified finite Galois extensions whose Galois group is isomorphic to the group $ UT_n(\mathbb{F}_p) $ of $ n \times n $ upper triangular unipotent matrices, or unitriangular matrices, over $ \mathbb{F}_p $. We have proved in our main theorem, Theorem \ref{tm001}, that when $3\leq n\leq p+2$, the largest upper ramification break is given by a linear expression of the valuation of the entries of a matrix explicitly constructed from the coefficients of a defining equation of the extension. This result is an analogue of the previous findings by Kanesaka and Sekiguchi \cite{KS} in the abelian case: they showed that the largest upper ramification number of an Artin-Schreier-Witt extension of $K$ defined by an equation $(x_0^p,\ldots,x_s^p)-(x_0,\ldots,x_s)=(a_0,\ldots,a_s)$ of Witt vectors is given by $\max_{0\leq i\leq s}(-v_K(a_i)p^{s-i})$, and the equation \eqref{eq107} proved in Theorem \ref{tm001} of this paper is a counterpart in the $UT_n(\mathbb{F}_p)$ case. Note that in the Kanesaka and Sekiguchi's result, the ramification number is derived only from the information of the elements $a_i$ of $K$. Similarly, the left-hand side of the equation \eqref{eq107} consists of the valuation of the element $s_{R,i,j}\in K_R$, which can be calculated without depending on any elements of $L$ not in $K$. Since every finite $ p $-group can be embedded into $ UT_n(\mathbb{F}_p) $ for some $ n $, findings in these cases are expected to shed light on more general situations.

This paper consists of five sections:

In Section \ref{sc002}, we define a valuation on the subring $FT_n(M)$ of the ring of upper triangular matrices over a discrete valuation field $M$ of positive characteristic, by extending the valuation on $M$. This valuation $v_M$ satisfies $v_M(XY)\geq \min(v_M(X),v_M(Y))$ for any $X,Y\in FT_n(M)$, analogous to the formula $v_M(x+y)\geq \min(v_M(x),v_M(y))$ for any $x,y\in M$. This valuation of matrices is a partial analogue to the filtration on Witt vectors in \cite{KS}.

In Section \ref{sc003}, we recall some properties of ramification numbers of separable extensions over a complete discrete valuation field $K$ of equal characteristic $p$. We also define a totally wildly ramified separable extension $K_R/K$ of degree $q$ with a unique upper ramification break $\frac{R}{q-1}$, where $R$ is a positive integer prime to $p$ and $q$ is a power of $p$.

In Section \ref{sc004}, we introduce a totally wildly ramified Galois extension $L/K$ with Galois group $G_{L/K}$ isomorphic to the group of unitriangular matrices $UT_n(\mathbb{F}_p)$ over $\mathbb{F}_p$. This extension has a defining equation of the form $X^{(p)}A=X$ with $A\in UT_n(K)$. This equation is analogous to the defining equation $x^p+a=x$ for Artin-Schreier cases. We denote by $\Theta$ one of the solutions of $X^{(p)}A=X$, where $X^{(p)}$ denotes the matrix obtained by raising each entry of $X$ to the $p$-th power. Specifically, the entries of $\Theta$ generate $L$ over $K$. We define an isomorphism $\iota_R:K\to K_R$ and extend it to separable closures. We also define the Galois extension $L_R/K_R$, where $L_R=\iota_R(L)$, and the matrices $\Theta_R=\iota_R(\Theta)\in UT_n(L_R)$ and $A_R=\iota_R(A)\in UT_n(K_R)$. In order to reduce the calculation of the ramification breaks to the Artin-Schreier case, we introduce another matrix $D_R$ defined by $A$, $A_R$, $\Theta$, and $\Theta_R$, and approximate it by a matrix $S_R$ constructed using only $A$ and $A_R$.

In Section \ref{sc005}, we study the relationship between the valuations of the entries $s_{R,i,j}$ of $S_R$ and the upper ramification breaks of the extension $L/K$. Specifically, our goal in this section is to prove that the formula
\begin{equation}\label{eq107}
	v_{K_R}(s_{R,i,j})=-qr_{i,j}+R
\end{equation}
holds for $1\leq i<j\leq n$ if $v_{K_R}(s_{R,l,m})<-q\mu_{l,m}+pR$ holds for all $i\leq l<m\leq j$, where $r_{i,j}$ is the largest upper ramification break of $K(\theta_{l,m}\ (i\leq l<m\leq j))/K$, and $\mu_{l,m}$ is a number calculated directly from the valuations of the entries of $A$. We also prove in Lemma \ref{lm006} that if \eqref{eq107} holds for $R=R'$ for some positive integer $R'$ prime to $p$, then it holds for $R>R'$ as well. This lemma is crucial for completing the proof of the main theorem in Section \ref{sc006}.

In Section \ref{sc006}, we complete the proof of the main theorem, which states that the formula \eqref{eq107} holds for $j-i\leq p+1$, by showing that $v_{K_R}(s_{R,l,m})<-q\mu_{l,m}+pR$ holds for all $i\leq l<m\leq j$ in this case. As a direct consequence, in Corollary \ref{cr003}, we also establish that $\mu_{i,j}$, which is easier to compute than $s_{R,i,j}$, can be used as an upper estimate for the upper ramification break, and provide a criterion for this estimate to be sharp. 

Generalization of our result to any $j-i\in \mathbb{Z}_{\geq 1}$ or to Galois extensions with different Galois group structures merits further investigation, but we expect that the method of this paper may be adapted to such cases.

Our approach is inspired by that of \cite{Ab}. Abrashkin constructed a filtration on some Lie algebra that corresponds to the ramification filtration of a profinite Galois extension $K_p/K$. The upper ramification groups of $ K_p/K $ are described in terms of a specific set of generators of the Galois group $ \Gal(K_p/K) $. This result allows one to deduce the ramification filtration of any Galois subextension of $ K_p/K $, provided that the Galois group of the subextension can be expressed using the same set of generators.
The field $K_R$ and the matrices $\Theta$, $\Theta_R$, $A$, $A_R$, $D_R$, and $S_R$ roughly correspond to $K'$, $\mathcal{F}$, $\mathcal{F}'$, $E$, $E'$, $A_1$, and $A'_1$ in \cite{Ab}.

The unique aspect of our approach is the use of the valuation $v_M$ on the ring $FT_n(M)$, which allows us to significantly simplify the calculation of the valuations of the entries of the matrices $S_R$ and $D_R$. Moreover, we show not only that the Artin-Schreier extension defined by $x^p+s_{R,i,j}=x$ over $K_R$ has a ramification break connected to the upper ramification breaks of $L/K$, but also that we have an explicit formula \eqref{eq107} relating $r_{i,j}$ and $s_{R,i,j}$ in Section \ref{sc005} and \ref{sc006}. This formula essentially provides an algorithm to compute the largest ramification break using only the entries of the matrices $A$ and $A_R$.

\section{Valuation on Upper Triangular Matrices}\label{sc002}

For any field $F$ of characteristic $p>0$, let $UT_n(F)$ (resp. $NT_n(F)$) denote the groups of $n\times n$ unipotent (resp. nilpotent) upper triangular matrices, or unitriangular (resp. niltriangular) matrices, over $F$, respectively.  We denote the identity matrix and the zero matrix of size $n\times n$ by $I$ and $O$, respectively. Let $FT_n(F)=\{iI+X\ \mid\ i\in \mathbb{F}_p,\ X\in NT_n(F)\}$, where $I$ denotes the identity matrix. The set $FT_n(F)$ is closed under matrix multiplication and subtraction, and contains the identity matrix. Thus, $FT_n(F)$ is a subring of $M_n(F)$. This ring $FT_n(F)$ consists of upper triangular matrices whose diagonal entries are the same element in $\mathbb{F}_p$. Note that $FT_n(F)^{\times}=FT_n(F)-NT_n(F)$, where $FT_n(F)^{\times}$ denotes the group of invertible elements in $FT_n(F)$. Denote the algebraic (resp. separable) closure of $F$ by $\overline{F}$ (resp. $F_{sep}$).

Note that if $X=(x_{i,j})_{i,j},Y=(y_{i,j})_{i,j}\in FT_n(F)$ and $XY=(z_{i,j})_{i,j}$, then we have
\begin{equation}\label{eq009}
	z_{i,j}=\sum_{l=i}^{j}x_{i,l}y_{l,j}
\end{equation}
for any $1\leq i<j\leq n$.
\begin{definition}\label{df007}
	Let $1\leq i<j\leq n$. We define a partition of length $s$ for $(i,j)$ as a sequence $\lambda=\{\lambda_{u}\}_{u=0,\ldots s}$ of integers satisfying
	\begin{equation}
		i=\lambda_0<\lambda_1<\cdots<\lambda_{s-1}<\lambda_s=j.
	\end{equation}
	We denote the length of a partition $\lambda$ by $\lvert\lambda\rvert $.
	Let $\Lambda_{i,j}$ be the set of all partitions for $(i,j)$.
\end{definition}

\begin{lemma}\label{lm019}
	Let $X=(x_{i,j})_{i,j}\in UT_n(M)$, and $X^{-1}=(x^*_{i,j})_{i,j}$.
	Then we have
	\begin{equation}\label{eq132}
		X^{-1}=\sum_{s=0}^{n-1}(-1)^s(X-I)^s,
	\end{equation}
	and
	\begin{equation}\label{eq118}
		x^*_{i,j}=\sum_{\lambda\in\Lambda_{i,j}} (-1)^{\lvert\lambda\rvert }\prod_{u=0}^{\lvert\lambda\rvert -1}x_{\lambda_u,\lambda_{u+1}}.
	\end{equation}
\end{lemma}
\begin{proof}
	Since $X-I$ is a nilpotent upper triangular matrix of size $n\times n$, we have $(X-I)^n=O$. By Taylor expansion, we have \eqref{eq132}. Hence, we have \eqref{eq118} by \eqref{eq009}.
\end{proof}

For any discrete valuation field $M$ of equal characteristic $p>0$, we denote the valuation, the ring of integers, and the maximal ideal of $M$ by $v_M$, $\mathcal{O}_{M}$, and $\mathfrak{m}_{M}$, respectively. We also define $\mathcal{O}_{\overline{M}}$ and $\mathfrak{m}_{\overline{M}}$ for $\overline{M}$ in the same manner.

\begin{lemma}\label{lm027}
	Let $\nu_{i,j}\ (1\leq i<j\leq n)$ be rational numbers satisfying $\nu_{i,j}\geq \nu_{i,l}+\nu_{l,j}$ for all $1\leq i<l<j\leq n$. Define subsets $H_{\leq \nu}$ and $H_{<\nu}$ of $UT_n(M)$ as follows:
	\begin{align}
		H_{\leq \nu} & = \{X=(x_{i,j})_{i,j}\in UT_n(M)\ \mid\ ^\forall 1\leq i<j\leq n,\ -v_{M}(x_{i,j})\leq \nu_{i,j}\}, \\
		H_{<\nu}     & = \{X=(x_{i,j})_{i,j}\in UT_n(M)\ \mid\ ^\forall 1\leq i<j\leq n,\ -v_{M}(x_{i,j})< \nu_{i,j}\}.
	\end{align}
	Then we have:
	\begin{enumerate}[label=(\roman*)]
		\item\label{ot058} The subset $H_{\leq \nu}$ is a subgroup of $UT_n(M)$.
		\item\label{ot059} The subset $H_{<\nu}$ is a subgroup of $UT_n(M)$.
		\item\label{ot060} The subset $H_{<\nu}$ is a normal subgroup of $H_{\leq \nu}$.
		\item\label{ot061} Let $X=(x_{i,j})_{i,j}\in H_{\leq \nu}$, $Y=(y_{i,j})_{i,j}\in H_{<\nu}$, $XY=(z_{i,j})_{i,j}$, and $YX=(z'_{i,j})_{i,j}$. Then we have
		\begin{equation}
			z_{i,j}\equiv x_{i,j} \equiv z'_{i,j} \pmod{\mathfrak{m}_{M}^{-\nu_{i,j}+1}}
		\end{equation}
		for all $1\leq i<j\leq n$.
	\end{enumerate}
\end{lemma}
\begin{proof}
	\begin{enumerate}[label=(\roman*)]
		\item Let $X=(x_{i,j})_{i,j},Y=(y_{i,j})_{i,j}\in H_{\leq \nu}$, $X^{-1}=(x^*_{i,j})_{i,j}$ and $XY=(z_{i,j})_{i,j}$. By \eqref{eq009}, we have
		      \begin{equation}
			      v_{M}(z_{i,j})\geq \min_{i\leq l<j}\left(v_{M}(x_{i,l})+v_{M}(y_{l,j})\right)
		      \end{equation}
		      for all $1\leq i<j\leq n$.
		      Since $X,Y\in H_{\leq \nu}$, we have $v_{M}(x_{i,l})\geq -\nu_{i,l}$ for all $1\leq i< l\leq j\leq n$ and $v_{M}(y_{l,j})\geq -\nu_{l,j}$ for all $1\leq i\leq l<j\leq n$. Thus, we obtain
		      \begin{equation}\label{eq133}
			      v_{M}(x_{i,l})+v_{M}(y_{l,j})\geq -(\nu_{i,l}+\nu_{l,j})\geq -\nu_{i,j}
		      \end{equation}
		      for all $1\leq i<l<j\leq n$ and
		      \begin{align}
			       & v_{M}(x_{i,j})+v_{M}(y_{j,j})\geq -\nu_{i,j}\label{eq134}  \\
			       & v_{M}(x_{i,i})+v_{M}(y_{i,j})\geq -\nu_{i,j}\label{eq135}.
		      \end{align}
		      Therefore we get
		      \begin{equation}
			      v_{M}(z_{i,j})\geq -\nu_{i,j},
		      \end{equation}
		      for all $1\leq i<j\leq n$, and consequently, $XY\in H_{\leq \nu}$.

		      Since we have
		      \begin{equation}
			      v_{M}\left(\prod_{u=0}^{\lvert\lambda\rvert -1}x_{\lambda_u,\lambda_{u+1}}\right)= \sum_{u=0}^{\lvert\lambda\rvert -1}v_{M}\left(x_{\lambda_u,\lambda_{u+1}}\right)\geq \sum_{u=0}^{\lvert\lambda\rvert -1}-\nu_{\lambda_u,\lambda_{u+1}}\geq -\nu_{i,j}
		      \end{equation}
		      for all $1\leq i<j\leq n$ and $\lambda\in\Lambda_{i,j}$, we deduce that $X^{-1}\in H_{\leq \nu}$ by Lemma \ref{lm019}. Therefore, the set $H_{\leq \nu}$ is a subgroup of $UT_n(M)$.

		\item This follows from the same argument as \ref{ot058}.

		\item Let $X=(x_{i,j})_{i,j}\in H_{\leq \nu}$, $Y=(y_{i,j})_{i,j}\in H_{<\nu}$ and $XYX^{-1}=(y'_{i,j})_{i,j}$. By \eqref{eq009}, we have
		      \begin{equation}
			      y'_{i,j}=\sum_{i\leq l\leq m\leq j}x_{i,l}y_{l,m}x^*_{m,j}=\sum_{i\leq l<m\leq j}x_{i,l}y_{l,m}x^*_{m,j}+\sum_{l=i}^{j}x_{i,l}x^*_{l,j}.
		      \end{equation}
		      Since $XX^{-1}=I$, we have
		      \begin{equation}
			      \sum_{l=i}^{j}x_{i,l}x^*_{l,j}=0
		      \end{equation}
		      for all $1\leq i<j\leq n$ by \eqref{eq009}. Thus, we obtain
		      \begin{equation}
			      y'_{i,j}=\sum_{i\leq l<m\leq j}x_{i,l}y_{l,m}x^*_{m,j}.
		      \end{equation}
		      Since we have
		      \begin{align*}
			      v_{M}(x_{i,l}y_{l,m}x^*_{m,j})&=v_{M}(x_{i,l})+v_{M}(y_{l,m})+v_{M}(x^*_{m,j})\\
						&> -\nu_{i,l}-\nu_{l,m}-\nu_{m,j}\\
						&\geq -\nu_{i,j}
			      \stepcounter{equation}\tag{\theequation}
		      \end{align*}
		      for all $1\leq i\leq l<m\leq j\leq n$, we deduce that $v_{M}(y'_{i,j})> -\nu_{i,j}$ for all $1\leq i<j\leq n$. Therefore, we have $XYX^{-1}\in H_{<\nu}$, and consequently, $H_{<\nu}$ is a normal subgroup of $H_{\leq \nu}$.

		\item Since $YX=X(X^{-1}YX)$ and $H_{<\nu}$ is a normal subgroup of $H_{\leq \nu}$, it suffices to show the assertion for $XY$. Since $Y\in H_{<\nu}$, the first inequality in \eqref{eq133} for all $1\leq i< l<j\leq n$ and the inequality in \eqref{eq135} are strict inequalities. Thus, we obtain the assertion for $XY$ by \eqref{eq009}.
	\end{enumerate}
\end{proof}

We will now generalize the valuation $v_{M}:M\to\mathbb{Z}\cup\{+\infty\}$ to $FT_n(M)\to \mathbb{Q}\cup\{+\infty\}$.
\begin{definition}\label{df001}
	Define maps $v_{M},\ \tilde{v}_{M}:FT_n(M)\to \mathbb{Q}\cup\{+\infty\}$ by
	\begin{align}
		v_{M}(X)         & = \min_{1\leq i<j\leq n}\left(\frac{v_{M}(x_{i,j})}{j-i}\right), \\
		\tilde{v}_{M}(X) & = \min_{\substack{1\leq i<j\leq n                                \\ (i,j)\neq (1,n)}}\left(\frac{v_{M}(x_{i,j})}{j-i}\right),
	\end{align}
	for any $X=(x_{i,j})_{i,j}\in FT_n(M)$.
\end{definition}
From this definition, for any $X=(x_{i,j})_{i,j}\in FT_n(M)$, we can assert that
\begin{align}
	v_{M}(X)         & =\sup\{r\in\mathbb{Q}\ \mid\ \forall i,j \in\{1,\ldots,n\},\ v_{M}(x_{i,j})\geq (j-i)r\},                            \\
	\tilde{v}_{M}(X) & =\sup\left\{r\in\mathbb{Q}\ \relmid\ 
	\begin{aligned}
		&\forall i,j \in\{1,\ldots,n\},\\
		 &(i,j)\neq (1,n)\Rightarrow v_{M}(x_{i,j})\geq (j-i)r
		\end{aligned}\right\},
\end{align}
since $v_{M}(0)=+\infty$ and $v_{M}(x_{i,i})\geq 0$ for $1\leq i\leq n$.

For any $X=(x_{i,j})_{i,j}\in FT_n(M)$ and $m\in \mathbb{Z}_{\geq 0}$, we denote by $X^{\left(p^m\right)}$ the matrix obtained by raising each entry of $X$ to the $p^m$-th power.

\begin{lemma}\label{lm001}
	Let $X=(x_{i,j})_{i,j}$ and $Y=(y_{i,j})_{i,j}$ be arbitrary elements of $FT_n(M)$.
	\begin{enumerate}[label=(\roman*)]
		\item\label{ot001} We have $v_{M}(X)=+\infty$ if and only if $X$ is diagonal.
		\item\label{ot003} We have $\tilde{v}_{M}(X)\geq v_{M}(X)$.
		\item\label{ot007} For any $i,j\in \mathbb{F}_p$, we have $v_{M}(iX+jY)\geq \min(v_{M}(X),v_{M}(Y))$ and $\tilde{v}_{M}(iX+jY)\geq \min(\tilde{v}_{M}(X),\tilde{v}_{M}(Y))$. In particular, adding integer multiples of $I$ does not affect $v_M$ and $\tilde{v}_{M}$.
		\item\label{ot008} we have $v_{M}\left(X^{(p)}\right)=pv_{M}(X)$ and $\tilde{v}_{M}\left(X^{(p)}\right)=p\tilde{v}_{M}(X)$.
		\item\label{ot004} We have $v_{M}(XY)\geq \min(v_{M}(X),v_{M}(Y))$ and $\tilde{v}_{M}(XY)\geq \min(\tilde{v}_{M}(X),\tilde{v}_{M}(Y))$.
		\item\label{ot005} If $X$ is nilpotent (i.e. not invertible), then $v_{M}(XY)\geq \min(v_{M}(X),\tilde{v}_{M}(Y))$ and $v_{M}(YX)\geq \min(v_{M}(X),\tilde{v}_{M}(Y))$.
		\item\label{ot006} If both $X$ and $Y$ are nilpotent, then $v_{M}(XY)\geq \min(\tilde{v}_{M}(X),\tilde{v}_{M}(Y))$.
		\item\label{ot002} If $X$ is invertible and $v_{M}(X)> v_{M}(Y)$, then we have $v_{M}(XY)=v_{M}(YX)=v_{M}(Y)$.
		\item\label{ot011} If $X$ is invertible and $\tilde{v}_{M}(X)> \tilde{v}_{M}(Y)$, then we have $\tilde{v}_{M}(XY)=\tilde{v}_{M}(YX)=\tilde{v}_{M}(Y)$.
		\item\label{ot009} Suppose $X$ is invertible. Then we have $v_{M}(X^{-1})=v_{M}(X)$ and $\tilde{v}_{M}(X^{-1})=\tilde{v}_{M}(X)$.
		\item\label{ot010} Take an integer $c$ and suppose $X$ is nilpotent. Fix $\gamma\in M$ satisfying $v_{M}(\gamma)=c$. Then we have
		\begin{align}
			 & v_{M}(\gamma X)\geq \min(\frac{c}{n-1},c)+v_{M}(X),                 \\
			 & \tilde{v}_{M}(\gamma X)\geq \min(\frac{c}{n-2},c)+\tilde{v}_{M}(X).
		\end{align}
	\end{enumerate}
\end{lemma}
\begin{proof}
	Let $XY=(z_{i,j})_{i,j}$.

	Assertions \ref{ot001}--\ref{ot008} immediately follow from the properties of the valuation of $M$ and the definitions of $v_{M}$ and $\tilde{v}_M$. We will prove the remaining assertions.
	\begin{enumerate}[label=(\roman*)]
		\addtocounter{enumi}{4}
		\item By \eqref{eq009}, we have
		      \begin{align*}
			      v_{M}(z_{i,j}) & \geq \min_{i\leq l\leq j}(v_{M}(x_{i,l})+v_{M}(y_{l,j})) \\
			                     & \geq \min_{i\leq l\leq j}((l-i)v_{M}(X)+(j-l)v_{M}(Y))   \\
			                     & =(j-i)\min(v_{M}(X),v_{M}(Y)).
			      \stepcounter{equation}\tag{\theequation}
		      \end{align*}
		      Thus we get $v_{M}(XY)\geq \min(v_{M}(X),v_{M}(Y))$ and $\tilde{v}_{M}(XY)\geq \min(\tilde{v}_{M}(X),\tilde{v}_{M}(Y))$.
		\item Since $X$ is nilpotent, we have $x_{i,i}=0$ for all $1\leq i\leq n$. Therefore,
		      \begin{align*}
			      v_{M}(z_{i,j}) & \geq \min_{i+1\leq l\leq j}(v_{M}(x_{i,l})+v_{M}(y_{l,j}))       \\
			                     & \geq \min_{i+1\leq l\leq j}((l-i)v_{M}(X)+(j-l)\tilde{v}_{M}(Y)) \\
			                     & =v_{M}(X)+(j-i-1)\min(v_{M}(X),\tilde{v}_{M}(Y))                 \\
			                     & \geq (j-i)\min(v_{M}(X),\tilde{v}_{M}(Y)).
			      \stepcounter{equation}\tag{\theequation}
		      \end{align*}
		      We thus obtain $v_{M}(XY)\geq \min(v_{M}(X),\tilde{v}_{M}(Y))$. The assertion for $YX$ follows from a similar argument.
		\item Since $X$ and $Y$ are nilpotent, we have $x_{i,i}=y_{j,j}=0$ for all $1\leq i<j\leq n$. Thus we have
		      \begin{align*}
			      v_{M}(z_{i,j}) & \geq \min_{i+1\leq l\leq j-1}(v_{M}(x_{i,l})+v_{M}(y_{l,j}))               \\
			                     & \geq \min_{i+1\leq l\leq j-1}((l-i)\tilde{v}_{M}(X)+(j-l)\tilde{v}_{M}(Y)) \\
			                     & =v_{M}(X)+v_{M}(Y)+(j-i-2)\min(\tilde{v}_{M}(X),\tilde{v}_{M}(Y))          \\
			                     & \geq (j-i)\min(\tilde{v}_{M}(X),\tilde{v}_{M}(Y))
			      \stepcounter{equation}\tag{\theequation}
		      \end{align*}
		      for all $1\leq i<j\leq n$ with $j-i>1$. We have $z_{i,i+1}=0$, so we also have $v_{M}(z_{i,j})\geq (j-i)\min(\tilde{v}_{M}(X),\tilde{v}_{M}(Y))$ for $1\leq i<j\leq n$ with $j-i=1$.
		      Thus we have $v_{M}(XY)\geq \min(\tilde{v}_{M}(X),\tilde{v}_{M}(Y))$.
		\item Let $(i,j)$ be one of the indices where $v_{M}(y_{i,j})=(j-i)v_{M}(Y)$. Then for all $i+1\leq l\leq j$, we get
		      \begin{align*}
			      v_{M}(x_{i,l}y_{l,j}) & =v_{M}(x_{i,l})+v_{M}(y_{l,j})  \\
			                            & \geq(l-i)v_{M}(X)+(j-l)v_{M}(Y) \\
			                            & >(j-i)v_{M}(Y)                  \\
			                            & =v_{M}(y_{i,j})                 \\
			                            & =v_M(x_{i,i}y_{i,j}).
			      \stepcounter{equation}\tag{\theequation}
		      \end{align*}
		      The last equality is due to $x_{i,i}\neq 0$, deduced from the assumption that $X$ is invertible. Therefore, by \eqref{eq009}, we obtain
		      \begin{equation}
			      v_{M}(z_{i,j})=v_M(x_{i,i}y_{i,j})=(j-i)v_{M}(Y).
		      \end{equation}
		      This gives $v_{M}(XY)\leq v_{M}(Y)$ by the definition of $v_{M}$. Combining this inequality with \ref{ot004}, we get $v_{M}(XY)=v_{M}(Y)$.
		      The assertion for $YX$ follows from a similar argument.
		\item This can be proved in the same way as \ref{ot002}.
		\item Since $X$ is invertible, $x_{1,1}\neq 0$ and $I-x_{1,1}^{-1}X$ is nilpotent. Then we have $(I-x_{1,1}^{-1}X)^n=O$, because the size of the matrix is $n\times n$. Hence, by Taylor expansion, it follows that
		      \begin{equation}
			      X^{-1}=x_{1,1}^{-1}\sum_{l=0}^{n-1}(I-x_{1,1}^{-1}X)^l.
		      \end{equation}
		      Applying \ref{ot007} and \ref{ot004}, we conclude that $v_{M}(X^{-1})\geq v_{M}(X)$. Applying this inequality to $X^{-1}$ gives $v_{M}(X)\geq v_{M}(X^{-1})$, and hence $v_{M}(X)= v_{M}(X^{-1})$. A similar argument shows that $\tilde{v}_{M}(X^{-1})=\tilde{v}_{M}(X)$.

		\item By the definition of $v_{M}$, we have
		      \begin{align*}
			      v_{M}(\gamma X) & =\min_{1\leq l<m\leq n}\left(\frac{v_{M}(\gamma x_{l,m})}{m-l}\right)                                               \\
			                      & =\min_{1\leq l<m\leq n}\left(\frac{c}{m-l}+\frac{v_{M}(x_{l,m})}{m-l}\right)                                        \\
			                      & \geq \min_{1\leq l<m\leq n}\left(\frac{c}{m-l}\right)+\min_{1\leq l<m\leq n}\left(\frac{v_{M}(x_{l,m})}{m-l}\right) \\
			                      & =\min(\frac{c}{n-1},c)+v_{M}(X).
			      \stepcounter{equation}\tag{\theequation}
		      \end{align*}
		      The assertion for $\tilde{v}_{M}$ can be proved in the same manner.
	\end{enumerate}
\end{proof}

\section{On Ramification Numbers}\label{sc003}

Let $K$ be a complete discrete valuation field with algebraically closed residue field $k$ of characteristic $p>0$.

For any finite separable extension $L/K$, we denote by $G_{L/K}$ the set of the embeddings of $L$ into $K_{sep}$ over $K$. Note that this set coincides with the Galois group when $L/K$ is Galois. We define the upper (resp. lower) ramification numbering $\{R_{L/K}^i\}_{i\in\mathbb{Q}_{\geq -1}}$ (resp. $\{R_{L/K,i}\}_{i\in\mathbb{Q}_{\geq -1}}$) on the data of equivalence relations over $G_{L/K}$,
the function $\phi_{L/K}$ translating the lower ramification numbering to the upper ramification numbering, and the inverse $\psi_{L/K}$ of $\phi_{L/K}$ as in \cite[Appendice]{De}. Note that the equivalence relation $R_{L/K}^i$ (resp. $R_{L/K,i}$) gets finer as $i$ increases.
When a rational number $i\geq -1$ satisfies $R_{L/K}^i\neq R_{L/K}^{i'}$ (resp. $R_{L/K,i}\neq R_{L/K,i'}$) for all $i'>i$, we call $i$ an upper (resp. lower) ramification break of $L/K$.
We denote the set of the upper (resp. lower) ramification breaks by $\mathcal{U}_{L/K}$ (resp. $\mathcal{L}_{L/K})$. Specifically, we define these sets as follows:
\begin{equation}
	\mathcal{U}_{L/K}=\left\{i\in\mathbb{Q}\ |\ ^\forall i'>i, R_{L/K}^i\neq R_{L/K}^{i'}\right\},
\end{equation}
\begin{equation}
	\mathcal{L}_{L/K}=\left\{i\in\mathbb{Q}\ |\ ^\forall i'>i, R_{L/K,i}\neq R_{L/K,i'}\right\}.
\end{equation}

Denote by $e_{L/K}$ the ramification index of $L/K$. We extend $v_{K}:K\to \mathbb{Z}\cup\{+\infty\}$ to $L\to\mathbb{Q}\cup\{+\infty\}$ by defining $v_{K}(x)=\frac{v_{L}(x)}{e_{L/K}}$ for any $x\in L$. We similarly extend $v_{K},\tilde{v}_{K}:FT_n(K)\to \mathbb{Q}\cup\{+\infty\}$ to $FT_n(L)\to\mathbb{Q}\cup\{+\infty\}$ by defining $v_{K}(X)=\frac{v_{L}(X)}{e_{L/K}}$ and $\tilde{v}_{K}(X)=\frac{\tilde{v}_{L}(X)}{e_{L/K}}$ for any $X\in FT_n(L)$.

When $L/K$ is Galois, we denote the $i$-th upper (resp. lower) ramification group of $G_{L/K}=\Gal(L/K)$ by $G_{L/K}^i$ (resp. $G_{L/K,i}$).

Let us now cite from \cite{De} the properties of ramification numbers which we will use in this paper.
\begin{lemma}(\cite[Appendice]{De})\label{lm010}
	For any separable extensions $M/K$ and $L/K$ of complete discrete valuation fields, we have the following assertions.
	\begin{enumerate}[label=(\roman*)]
		\item \label{ot019}
		      Suppose $M\supset L$ and fix an embedding $\tau_L$ of $L$ into $K_{sep}=L_{sep}$. Then for any $\tau_1,\tau_2:M\to K_{sep}\in G_{M/K}$ such that $\tau_1|_L=\tau_2|_L=\tau_L$ and $i\in\mathbb{Q}_{\geq -1}$, we have
		      \begin{equation}
			      \tau_1\equiv \tau_2\pmod{R_{M/L,i}}\ \Leftrightarrow\ \tau_1\equiv \tau_2\pmod{R_{M/K,i}}.
		      \end{equation}
		      Consequently, we have $\mathcal{L}_{M/K}\supset \mathcal{L}_{M/L}$.
		\item \label{ot020}
		      Suppose $M\supset L$. Then for any $\tau_1,\tau_2:M\to K_{sep}\in G_{M/K}$ and $i\in\mathbb{Q}_{\geq -1}$, we have
		      \begin{align*}
			      &\tau_1|_L \equiv \tau_2|_L  \pmod{R_{L/K}^i}\ \\&\qquad\Leftrightarrow\ ^\exists\tau'_1,\tau'_2\in G_{M/K}\ \textrm{s.t.}\left\{ \begin{aligned}
				       & \tau'_1|_L=\tau_1|_L                    \\
				       & \tau'_2|_L=\tau_2|_L                    \\
				       & \tau'_1 \equiv \tau'_2 \pmod{R_{M/K}^i}
			      \end{aligned}\right..
			      \stepcounter{equation}\tag{\theequation}
		      \end{align*}
		      Consequently, we have $\mathcal{U}_{M/K}\supset \mathcal{U}_{L/K}$.
		\item \label{ot050}
		      Suppose $M\supset L$. Then for any $\tau_1,\tau_2:M\to K_{sep}\in G_{M/K}$ and $i\in\mathbb{Q}_{\geq -1}$, we have
		      \begin{align*}
			      &\tau_1|_L \equiv \tau_2|_L  \pmod{R_{L/K,\phi_{M/L}(i)}}\ \\&\qquad\Leftrightarrow\ ^\exists\tau'_1,\tau'_2\in G_{M/K}\ \textrm{s.t.}\left\{ \begin{aligned}
				       & \tau'_1|_L=\tau_1|_L                    \\
				       & \tau'_2|_L=\tau_2|_L                    \\
				       & \tau'_1 \equiv \tau'_2 \pmod{R_{M/K,i}}
			      \end{aligned}\right..
			      \stepcounter{equation}\tag{\theequation}
		      \end{align*}
		      Consequently, we have $\phi_{M/L}(\mathcal{L}_{M/K})\supset \mathcal{L}_{L/K}$.
		\item \label{ot021}
		      Suppose $M\supset L$. Then we have
		      \begin{equation}
			      \phi_{M/K}=\phi_{L/K}\circ\phi_{M/L}
		      \end{equation}
		      and
		      \begin{equation}
			      \psi_{M/K}=\psi_{M/L}\circ\psi_{L/K}.
		      \end{equation}
		\item \label{ot038}
		      We have
		      \begin{equation}
			      \mathcal{L}_{L/K}=\{l\in \mathbb{Q}_{>-1}\ \mid\ \phi_{L/K}(x)\textrm{ is non-differentiable at }x=l\}
		      \end{equation}
		      and
		      \begin{equation}
			      \mathcal{U}_{L/K}=\{r\in \mathbb{Q}_{>-1}\ \mid\ \psi_{L/K}(x)\textrm{ is non-differentiable at }x=r\}
		      \end{equation}
		\item \label{ot047}
		      We have
		      \begin{equation}
			      \max\mathcal{U}_{LM/K}=\max(\mathcal{U}_{L/K}\cup\mathcal{U}_{M/K}).
		      \end{equation}
	\end{enumerate}
\end{lemma}
\begin{proof}
	\ref{ot019}--\ref{ot021} See \cite[Appendice, Section A.4]{De}.
	\begin{enumerate}[label=(\roman*)]

		\addtocounter{enumi}{4}
		\item This follows from \cite[Appendice, Proposition A.4.2]{De}, \ref{ot021}, and the definitions of $\mathcal{L}_{L/K}$ and $\mathcal{U}_{L/K}$.
		\item This follows from \ref{ot020}.

	\end{enumerate}

\end{proof}

\begin{lemma}\label{lm025}
	Let $M/L/K$ be a tower of finite totally ramified separable extensions and let $K'/K$ be a finite totally ramified separable extension of degree $e$. Suppose that the following conditions hold.
	\begin{enumerate}[label=(\alph*)]
		\item\label{ot044} $\mathcal{U}_{K'/K}$ is a singleton $\{r'\}$.
		\item\label{ot045} $\mathcal{L}_{M/L}$ is a singleton $\{\psi_{M/K}(r)\}$
		\item\label{ot046} $\max(\mathcal{U}_{L/K})<r'<r$.
	\end{enumerate}
	Then we have
	\begin{equation}\label{eq092}
		\max\mathcal{U}_{LK'/K'}\leq r'.
	\end{equation}
	and
	\begin{equation}\label{eq091}
		\max\mathcal{U}_{MK'/K'}=er-(e-1)r'.
	\end{equation}
\end{lemma}
\begin{proof}
	Note that we have
	\begin{equation}
		\psi_{K'/K}(x)=\left\{\begin{aligned}
			 & x          & \left(-1\leq x<r'\right) \\
			 & ex-(e-1)r' & \left(x\geq r'\right)
		\end{aligned}\right..
	\end{equation}
By Lemma \ref{lm010}\ref{ot038}, condition \ref{ot045} and the definition that $\psi_{M/L}$ is the inverse of $\phi_{M/L}$, we get that $\psi_{M/L}(x)$ has a unique non-differentiable point at $x=\phi_{M/L}(\psi_{M/K}(r))$. By Lemma \ref{lm010}\ref{ot047}, we have $\phi_{M/L}\circ\psi_{M/K}=\psi_{L/K}$, and hence $\psi_{M/L}(x)$ has a unique non-differentiable point at $x=\psi_{L/K}(r)$. This implies 
\begin{equation}
	\mathcal{U}_{M/K}=\{r\}\cup \mathcal{U}_{L/K}
\end{equation}
by Lemma \ref{lm010}\ref{ot038}.
	Therefore, by Lemma \ref{lm010}\ref{ot047}, we have
	\begin{equation}
		\max\mathcal{U}_{MK'/K}=\max(\mathcal{U}_{M/K}\cup \mathcal{U}_{K'/K})=\max(\{r,r'\}\cup\mathcal{U}_{L/K})=r.
	\end{equation}
	By condition \ref{ot046} and Lemma \ref{lm010}\ref{ot021}, we get
	\begin{equation}
		\psi_{MK'/K}(x)=\left\{\begin{aligned}
			 & \psi_{MK'/K'}(x)          & \left(-1\leq x<r'\right) \\
			 & \psi_{MK'/K'}(ex-(e-1)r') & \left(x\geq r'\right)
		\end{aligned}\right..
	\end{equation}
	Therefore by Lemma \ref{lm010}\ref{ot038} and condition \ref{ot046}, we get \eqref{eq091}. By the same argument, we obtain \eqref{eq092}.
\end{proof}

Suppose that $K$ is of equal characteristic $p>0$ and fix $t\in K$ such that $v_{K}(t)=-1$ for the rest of this paper.
We will now introduce a separable extension $K_R/K$ of $K$. Using this extension, we will show in Lemma \ref{lm022} that the calculation of the ramification numbers of finite non-abelian Galois extensions of $K$ can be reduced to the case of Artin-Schreier extensions.

\begin{definition}\label{df003}
	Let $R$ be a positive integer prime to $p$.
	Let $K_R=K(T)$ be the extension of $K$ defined by
	\begin{equation}\label{eq104}
		{T}^{q}+{T}^{q-1}=t^{R},
	\end{equation}
	where $q$ is a power of $p$.
\end{definition}
The equation \eqref{eq104} yields
\begin{equation}\label{eq059}
	\left(\frac{t^{\frac{R}{q-1}}}{T}\right)^{q}-\frac{t^{\frac{R}{q-1}}}{T}=t^{\frac{R}{q-1}}
\end{equation}
in $\overline{K}$.
Therefore, $K_{R}\left(t^{\frac{1}{q-1}}\right)$ is an Artin-Schreier extension of $K\left(t^{\frac{1}{q-1}}\right)$, and consequently, a separable extension of $K$.

We have $v_{K(t^{\frac{1}{q-1}})}(t^{\frac{R}{q-1}})=-R$. Since $-R$ is prime to $p$ and negative, by \eqref{eq059}, we have $[K_R:K]=[K_{R}(t^{\frac{1}{q-1}}):K(t^{\frac{1}{q-1}})]=q$.

\begin{lemma}\label{lm007}
	We have $\mathcal{U}_{K_{R}/K}=\left\{\frac{R}{q-1}\right\}$.
\end{lemma}
\begin{proof}
	It follows from \eqref{eq059} and the fact that $K(t^{\frac{1}{q-1}})/K$ is totally tamely ramified.
\end{proof}

Let $t_R\in \overline{K}$ be the $qR$-th root of $t^R\left(1+T^{-1}\right)^{-1}$ congruent to $t^{\frac{1}q}$ modulo $t^{\frac{1}q}\mathfrak{m}_{\overline{K}}$, where $\mathfrak{m}_{\overline{K}}$ denotes the maximal ideal of $\overline{K}$. Then we have $t_R^R=T$, since
\begin{equation}
	t^R\left(1+T^{-1}\right)^{-1}=T^q
\end{equation}
by the definition of $T$.
\begin{lemma}\label{lm013}
	We have $K_{R}=K(t_{R})$ and
	\begin{equation}\label{eq032}
		t^l =t_R^{ql}\left(1+\sum_{i=1}^{\infty}\binom{\frac{l}{R}}{i}t_R^{-iR}\right)
	\end{equation}
	for $l\in\mathbb{Z}$.
\end{lemma}
\begin{proof}
	The inclusion $K_{R}=K(T)\subset K(t_{R})$ follows from $t_R^R=T$. We will show the reverse inclusion.

	By Hensel's lemma and the fact that the residue field of $K_{R}$ is algebraically closed, it follows that the $R$-th roots of $1+T^{-1}$ and their inverses belong to $K_{R}$.
	Let $z$ be the $R$-th root of $1+T^{-1}$ congruent to 1 modulo $\mathfrak{m}_{K_R}$. Thus, by the definition of $t_R$, we have
	\begin{equation}\label{eq066}
		t_R^q=tz^{-1}\in K_R.
	\end{equation}
	Take $c_1,c_2\in \mathbb{Z}$ such that $qc_1+Rc_2=1$.
	Then we have
	\begin{equation}
		t_R=\left(t_R^q\right)^{c_1}\left(t_R^R\right)^{c_2}=\left(tz^{-1}\right)^{c_1}T^{c_2}\in K_R.
	\end{equation}
	By \eqref{eq066}, we have
	\begin{equation}
		t=t_R^qz.
	\end{equation}
	By the definition of $z$ and binomial expansion, we get \eqref{eq032}.
\end{proof}

Let $\iota_R:K\to K_R$ be the ring isomorphism defined by $t\mapsto t_R$ and $u \mapsto u^{\frac{1}{q}}$ for any $u\in k$. Since $K_R/K$ is a separable extension, we can choose a common separable closure for both fields so that $K_{sep}=K_{R,sep}$. We then take an extension $\iota_R:K_{sep}\to K_{R,sep}=K_{sep}$ of $\iota_R$ to this separable closure.

\begin{lemma}\label{lm022}
	Let $L/L'/K$ be a tower of finite separable totally ramified extensions and let $L_R=\iota_R(L)$ and $L'_R=\iota_R(L')$. Let $K'_R/K_R$ be a separable extension. Suppose that the following conditions hold.
	\begin{enumerate}[label=(\alph*)]
		\item\label{ot035} $LL_R=L'L_RK'_R$.
		\item\label{ot037} $\mathcal{L}_{L/L'}$ is a singleton $\{\psi_{L/K}(r)\}$
		\item\label{ot036} $q\max(\mathcal{U}_{L'/K})<R<(q-1)r$.
	\end{enumerate}
	Then we have
	\begin{equation}
		\max\mathcal{U}_{K'_{R}/K_R}=qr-R.
	\end{equation}
\end{lemma}
\begin{proof}
	By Lemma \ref{lm010}\ref{ot047}, we have
	\begin{equation}
		\max\mathcal{U}_{LL_R/K_R}=\max(\mathcal{U}_{LK_R/K_R}\cup \mathcal{U}_{L_R/K_R})=\max(\mathcal{U}_{LK_R/K_R}\cup\mathcal{U}_{L/K}).
	\end{equation}
	By Lemma \ref{lm025}, we get
	\begin{equation}
		\max\mathcal{U}_{LK_R/K_R}=qr-R
	\end{equation}
	By Lemma \ref{lm010}\ref{ot050} and condition \ref{ot037} and \ref{ot036}, we obtain
	\begin{equation}\label{eq093}
		\max\mathcal{U}_{L/K}=r.
	\end{equation}
	Therefore, we deduce that
	\begin{equation}\label{eq094}
		\max\mathcal{U}_{LL_R/K_R}=\max(r,qr-R)=qr-R.
	\end{equation}

	On the other hand, we have
	\begin{align*}
		\max\mathcal{U}_{LL_R/K_R} & =\max(\mathcal{U}_{L'K_R/K_R}\cup \mathcal{U}_{L_R/K_R}\cup\mathcal{U}_{K'_R/K_R}) \\
		                           & =\max(\mathcal{U}_{L'K_R/K_R}\cup \mathcal{U}_{L/K}\cup\mathcal{U}_{K'_R/K_R})     \\
		                           & =\max(\mathcal{U}_{L'K_R/K_R}\cup\mathcal{U}_{K'_R/K_R}\cup \{r\}).
		\stepcounter{equation}\tag{\theequation}
	\end{align*}
	by condition \ref{ot035} and \eqref{eq093}.
	By Lemma \ref{lm025}, we have
	\begin{equation}
		\max\mathcal{U}_{L'K_R/K_R}\leq \frac{R}{q-1}<r.
	\end{equation}
	Therefore, we get
	\begin{equation}\label{eq095}
		\max\mathcal{U}_{LL_R/K_R}=\max(\mathcal{U}_{K'_R/K_R}\cup \{r\}).
	\end{equation}
	Comparing \eqref{eq094} and \eqref{eq095}, we get the assertion.

\end{proof}

\section{Galois Extensions with Galois Group Isomorphic to Unitriangular Matrix Group over $\mathbb{F}_p$}\label{sc004}

Let $K$ be a complete discrete valuation field of equal characteristic $p>0$ with algebraically closed residue field $k$ and $L$ be a totally wildly ramified finite Galois extension of $K$ such that $\Gal(L/K)\simeq UT_n(\mathbb{F}_p)$ for some $n\geq 2$. We identify the residue field of $L$ with $k$. Define $K_R$ as in Definition \ref{df003} for any positive integer $R$ prime to $p$. Define $\mathcal{L}_{M_1/M_2}$ and $\mathcal{U}_{M_1/M_2}$ for any finite separable extension $M_1/M_2$ of complete discrete valuation fields as in Section \ref{sc003}.

\begin{definition}\label{df002}
	Every $x\in K$ can be uniquely written in the form $x=\sum_{i=v_{K}(x)}^{\infty}x_{i}t^{-i}$ with $x_i\in k$ for all $i$. Define $\coef_{K,i}:K\to k$ as \begin{equation}
		\coef_{K,i}(x)=x_i
	\end{equation}
	for all $i\in \mathbb{Z}$.
\end{definition}

\begin{lemma}\label{lm011}
	There exists a matrix $A=(a_{i,j})_{1\leq i,j\leq n}\in UT_n(K)$ satisfying the following conditions.
	\begin{enumerate}[label=(\alph*)]
		\item\label{ot029} For all $1\leq i<j\leq n$ and $l\in p\mathbb{Z}\cup\mathbb{Z}_{\geq 0}$,
		\begin{equation}
			\coef_{K,l}(a_{i,j})=0.
		\end{equation}
		\item\label{ot030} The equation $X^{(p)}A=X$ defines $L/K$. In other words, if we take a solution $\Theta=(\theta_{i,j})_{1\leq i,j\leq n}\in UT_n(K_{sep})$ of this equation, then $L=K(\theta_{i,j}\ (1\leq i<j\leq n))$.
	\end{enumerate}
\end{lemma}
\begin{proof}
	By the same argument as in \cite[Lemma 1.9(a)]{Im}, we see that there exists a matrix $A$ satisfying condition \ref{ot030}.
	Since $X^{(p)}B^{(p)}AB^{-1}=X$ for any matrix $B\in UT_{n}(K)$ defines the same extension as $L/K$, we may replace $A$ with $B^{(p)}AB^{-1}$.

	Let $E_{i,j}$ be the matrix unit with a 1 at the $(i,j)$-th entry for $1\leq i<j\leq n$.

	Let $1\leq i<j\leq n$ and $x\in K$. Denote the matrix $(I+xE_{i,j})^{(p)}A(I+xE_{i,j})^{-1}$ by $A'=(a'_{l,m})_{l,m}$. Then we have
	\begin{equation}
		a'_{l,m}=\left\{\begin{aligned}
			 & a_{l,m}+x^p-x      & ((l,m)=(i,j))      \\
			 & a_{l,m}+x^pa_{j,m} & (l=i, m>j)         \\
			 & a_{l,m}-xa_{l,i}   & (l<i, m=j)         \\
			 & a_{l,m}            & (\text{otherwise})
		\end{aligned}\right..
	\end{equation}
	This gives
	\begin{equation}
		(l,m)\neq (i,j)\wedge m-l\leq j-i \Rightarrow\ a'_{l,m}=a_{l,m}.
	\end{equation}
	Note that for any $y\in K$, there exists $x\in K$ such that $\coef_{K,l}(y+x^p-x)=0$ holds for all $l\in p\mathbb{Z}\cup\mathbb{Z}_{\geq 0}$, since the residue field $k$ is algebraically closed. Hence, we can make $A$ satisfy condition \ref{ot029} by repeatedly replacing $A$ by $(I+xE_{i,j})^{(p)}A(I+xE_{i,j})^{-1}$ for some appropriate $x\in K$ for each $1\leq i<j\leq n$ in ascending order of indices $(i,j)$.
\end{proof}

Fix $A=(a_{i,j})_{i,j}$ satisfying the conditions of Lemma \ref{lm011} for the rest of this paper. Denote the $i,j$-th entry of $A^{-1}$ by $a_{i,j}^*$ for any $1\leq i,j\leq n$.
Denote $K(\theta_{l,m}\ (i\leq l<m\leq j))$ by $L_{i,j}$ for any $1\leq i<j\leq n$.

\begin{lemma}\label{lm004}
	For any $1\leq i<j\leq n$, we have $\mathcal{L}_{L_{i,j}/L_{i,j-1}L_{i+1,j}}=\{\max\mathcal{L}_{L_{i,j}/K}\}$ and $\mathcal{U}_{L_{i,j}/K}-\mathcal{U}_{L_{i,j-1}L_{i+1,j}/K}=\{\max\mathcal{U}_{L_{i,j}/K}\} $.
\end{lemma}

\begin{proof}
	Let $l=j-i$.
	We prove this lemma by induction on $l$.

	Suppose $l=1$, then the assertion clearly holds, because $L_{i,j}/K$ is an Artin-Schreier extension and $\#\mathcal{U}_{L_{i,j}/K}=1$.

	Suppose $l\geq 2$. Let $r=\max\mathcal{L}_{L_{i,j}/K}$ and denote $G=G_{L_{i,j}/K}$. Note that we have $G=G_1$, since $L_{i,j}/K$ is totally wildly ramified. By \cite[IV, \S 2, Corollaire 1 of Proposition 9]{Se}, we get
	\begin{equation}
		[G_{r},G]\subset G_{r+1}=\{1\}.
	\end{equation}
	Thus we have $G_{r}\subset Z(G)$. Note that $G\simeq UT_{j-i+1}(\mathbb{F}_p)$ and $G_{L_{i,j}/L_{i,j-1}L_{i+1,j}}=Z(G)$.
	Hence, $Z(G)$ is isomorphic to $\mathbb{Z}/p\mathbb{Z}$. Thus it follows that $G_{r}=G_{L_{i,j}/L_{i,j-1}L_{i+1,j}}$.
	This yields $\mathcal{L}_{L_{i,j}/L_{i,j-1}L_{i+1,j}}=\{\max\mathcal{L}_{L_{i,j}/K}\}$ and $\mathcal{U}_{L_{i,j}/K}-\mathcal{U}_{L_{i,j-1}L_{i+1,j}/K}=\{\max\mathcal{U}_{L_{i,j}/K}\} $.
\end{proof}

Let $r_{i,j}=\max\mathcal{U}_{L_{i,j}/K}$ for any $1\leq i<j\leq n$. For convenience, we set $r_{i,i}=-1$ for any $1\leq i\leq n$. 

\begin{corollary}\label{cr002}
	We have $r_{i,j}>r_{l,m}$ for any $1\leq i\leq l<m\leq j\leq n$ satisfying $(i,j)\neq (l,m)$.
\end{corollary}
\begin{proof}
	It follows immediately from Lemma \ref{lm004} and the definition of $r_{i,j}$.
\end{proof}

Let $m_{A}=-v_{K}(A)$. Fix an integer $N>\log_{p}(n\left(p^{\frac{n(n-1)}{2}}+1\right)m_{A}+p)+\frac{n(n-1)}{2}$ and let $q=p^N$ for the rest of this paper. Let $R$ be a positive integer prime to $p$. Define $K_R=K(T)$ to be the extension of $K$ defined by $t^{R}={T}^{q}+{T}^{q-1}$ as in Definition \ref{df003}.

\begin{lemma}\label{lm008}
	For any $1\leq i<j\leq n$ such that $j-i>1$, we have \begin{equation}
		(q-1)r_{i,j}-q\max(r_{i,j-1},r_{i+1,j})>nm_{A}+p.
	\end{equation}
\end{lemma}
\begin{proof}
	By the definition of upper ramification groups, we have
	\begin{equation}
		\mathcal{U}_{L_{i,j}/K}\in \lvert\Gal(L_{i,j}/K)\rvert ^{-1}\mathbb{Z}=p^{-\frac{(j-i)(j-i+1)}{2}}\mathbb{Z}.
	\end{equation}
	Thus $r_{i,j}>\max(r_{i,j-1},r_{i+1,j})$ implies $r_{i,j}-\max(r_{i,j-1},r_{i+1,j})\geq p^{-\frac{(j-i)(j-i+1)}{2}}$.

	Since $q>p^{\frac{n(n-1)}{2}}(n\left(p^{\frac{n(n-1)}{2}}+1\right)m_{A}+p)$ by definition, we get
	\begin{equation}
		q(r_{i,j}-\max(r_{i,j-1},r_{i+1,j}))> n\left(p^{\frac{n(n-1)}{2}}+1\right)m_{A}+p.
	\end{equation}

	By Lemma \ref{lm010}\ref{ot020} and Lemma \ref{lm004}, we have $r_{i,j}\leq r_{1,n}$. By definition, we have $\psi_{L_{i,j}/K}(x)\geq x$ for any $x\geq -1$. Thus we have $r_{1,n}\leq \max\mathcal{L}_{L_{1,n}/K}$.
	By Lemma \ref{lm004}, we have $\mathcal{L}_{L_{1,n}/L_{1,n-1}L_{2,n}}=\{\max\mathcal{L}_{L_{1,n}/K}\}$.
	Note that $L_{1,n}/L_{1,n-1}L_{2,n}$ is an extension defined by an Artin-Schreier polynomial whose roots include $\theta_{1,n}$. Therefore, $\max\mathcal{L}_{L_{1,n}/K}\leq -v_{L_{1,n}}(\theta_{1,n})$. Hence, by Lemma \ref{lm001}, we obtain
	\begin{align*}
		r_{i,j} & \leq -v_{L_{1,n}}(\theta_{1,n})      \\
		        & \leq -(n-1)v_{L_{1,n}}(\Theta)       \\
		        & =-\frac{n-1}{p}v_{L_{1,n}}(A)        \\
		        & =-\frac{n-1}{p}e_{L_{1,n}/K}v_{K}(A) \\
		        & =(n-1)p^{\frac{n(n-1)}{2}-1}m_{A}    \\
		        & <np^{\frac{n(n-1)}{2}}m_{A}.
		\stepcounter{equation}\tag{\theequation}
	\end{align*}
	Therefore we have
	\begin{align*}
		(q-1)r_{i,j}-q\max(r_{i,j-1},r_{i+1,j}) & =q(r_{i,j}-\max(r_{i,j-1},r_{i+1,j}))-r_{i,j}                           \\
		                                        & >n\left(p^{\frac{n(n-1)}{2}}+1\right)m_{A}+p-np^{\frac{n(n-1)}{2}}m_{A} \\
		                                        & =nm_{A}+p.
		\stepcounter{equation}\tag{\theequation}
	\end{align*}
\end{proof}

\begin{definition}\label{df004}
	Take a common separable closure $K_{sep}=K_{R,sep}$ of $K$ and $K_R$, and let $\iota_R:K_{sep}\to K_{R,sep}=K_{sep}$ be an extension of the ring isomorphism $\iota_R:K\to K_{R}$ defined by $t\mapsto t_R$ and $u\mapsto u^{\frac{1}{q}}$ for any $u\in k$ as in Section \ref{sc003}.
	Define $P,\ \eta_{R}:K_{sep}\to K_{R,sep}=K_{sep}$ by $P(x)=x^{p}-x$ and $\eta_R(x)=x-{\iota_R(x)}^{q}$ for any $x\in K_{sep}$.

	For any $X\in M_n(K_{sep})$, we denote by $P(X)$, $\iota_R(X)$ and $\eta_R(X)$ the matrices obtained by applying $P$, $\iota_R$ and $\eta_R$ respectively to each entry of $X$. Denote $K_R(\iota_R(\theta_{l,m})\ (i\leq l<m\leq j))$ by $L_{R,i,j}$.
\end{definition}

\begin{definition}\label{df005}
	We define matrices $\Theta_R=(\theta_{R,i,j})_{i,j}$, $A_R=(a_{R,i,j})_{i,j}$, $W_{R,e}=(w_{R,e,i,j})_{i,j}$, $\Delta_R=(\delta_{R,i,j})_{i,j}$, $D_R=(d_{R,i,j})_{i,j}$ as follows:
	\begin{align}
		 & \Theta_R=\iota_R(\Theta) \in UT_n(L_{R}),                  &                  \\
		 & A_R=\iota_R(A)\in UT_n(K_R) ,                              &                  \\
		 & W_{R,0}=I ,                                                &                  \\
		 & W_{R,e}=A_{R}^{\left(p^{e-1}\right)}W_{R,e-1}\in UT_n(K_R) & (1\leq e\leq N),
		\label{eq004}                                                                    \\
		 & \Delta_R={\Theta_R}-\Theta{W_{R,N}}\in NT_n(LL_{R}),       &                \label{eq137}  \\
		 & D_R=P(\Delta_R) \in NT_n(LL_{R}).                          &
	\end{align}
	Denote the $(i,j)$-th entry of $A_R^{-1}$ by $a_{R,i,j}^*$ for any $1\leq i,j\leq n$.
\end{definition}
The matrices $\Delta_R$ and $D_R$ are nilpotent and the others are unipotent. Note that the matrices with subscript $R$ depend on $R$.

Note that we have
\begin{align}
	 & \Theta_R^{(p)}A_R=\Theta_R,                                                                 & \label{eq082}                  \\
	 & \Theta_R^{\left(p^{i}\right)}=\Theta_R^{\left(p^{i+j}\right)}{W_{R,j}}^{\left(p^{i}\right)} & (i\geq 0,\ 0\leq j\leq N),     \\
	 & W_{R,e}=W_{R,e-1}^{(p)}A_R                                                                  & (1\leq e\leq N), \label{eq062} \\
	 & \Delta_R=-\eta_R(\Theta)W_{R,N}.                                                            &
	\label{eq011}
\end{align}

\begin{remark}\label{rm001}
	We have
	\begin{equation}\label{eq046}
		L_{i,j}K_{R}(\delta_{R,l,m}\ (i\leq l<m\leq j))=L_{R,i,j}(\delta_{R,l,m}\ (i\leq l<m\leq j))=L_{i,j}L_{R,i,j}
	\end{equation}
	by the definition of $\Delta_R$, and the fact that $\theta_{i,j}\in L_{i,j}$, $\theta_{R,i,j}\in L_{R,i,j}$ and $W_{R,N}\in UT_n(K_R)$. Thus the matrix $\Delta_R$ connects the extension $L/K$ with the extension $L_{R}/K_R$. Note that $D_R=P(\Delta_R)$ and that $D_R$ is not expressed in terms of the coefficients of the defining equation of $L/K$, since the definition of $D_R$ includes $\Theta$ and $\Theta_R$.

	We will now show that if we have an extension $K'_R/K_R$ satisfying the conditions of Lemma \ref{lm022}, we can extract the information of the ramification breaks of $L/K$ by investigating $K'_R/K_R$. We will see that we can construct this extension by approximating $D_R$ by a matrix in $NT_n(K_R)$.
\end{remark}

\begin{lemma}\label{lm005}
	Let $1\leq i<j\leq n$ and $R$ be an integer prime to $p$ such that $q\max(r_{i,j-1},r_{i+1,j})<R<(q-1)r_{i,j}$. Assume that we have $x\in K_{R}$ satisfying $d_{R,i,j}-x\in \mathcal{O}_{L_{i,j}L_{R,i,j}}$. Then there exist $y,z\in K_{R}$ such that $x=y+P(z)$ and $v_{K_R}(y)=-qr_{i,j}+R$. Consequently, $v_{K_R}(x)\leq-qr_{i,j}+R<0$ and the equality holds if $p\nmid v_{K_R}(x)$.
\end{lemma}
\begin{proof}
	We prove this lemma by applying Lemma \ref{lm022} with $L=L_{i,j}$, $L'=L_{i,j-1}L_{i+1,j}$, and $K'_R=K_R(\xi)$, where $\xi\in K_{R,sep}=K_{sep}$ satisfies $P(\xi)=x$. The conditions \ref{ot037} and \ref{ot036} are satisfied by assumption. Therefore it suffices to show
	\begin{equation}\label{eq138}
		L_{i,j}L_{R,i,j}=L_{i,j-1}L_{i+1,j}L_{R,i,j}K_R(\xi).
	\end{equation}

 First, we will show that $d_{R,i,j}\in L_{i,j-1}L_{i+1,j}L_{R,i,j}$. Since $\Theta$ satisfies $\Theta=\Theta^{(p)}A$, we have
 \begin{equation}
	\theta_{i,j}=\theta_{i,j}^p+a_{i,j}+\sum_{l=i+1}^{j-1} \theta_{i,l}^pa_{l,j} \equiv \theta_{i,j}^p \pmod{L_{i,j-1}L_{i+1,j}}.
 \end{equation}
 Hence, we have $P(\theta_{i,j})\in L_{i,j-1}L_{i+1,j}$. By \eqref{eq137}, we have
 \begin{align*}
	\delta_{R,i,j}=\theta_{R,i,j}-\theta_{i,j}-w_{R,N,i,j}-\sum_{l=i+1}^{j-1} \theta_{i,l}w_{R,N,l,j}&\equiv -\theta_{i,j} \\&\pmod{L_{i,j-1}L_{i+1,j}L_{R,i,j}}.
			      \stepcounter{equation}\tag{\theequation}
		      \end{align*}
   Therefore, we get $d_{R,i,j}=P(\delta_{R,i,j})\in L_{i,j-1}L_{i+1,j}L_{R,i,j}$. Furthermore, we obtain $d_{R,i,j}-x\in\mathcal{O}_{L_{i,j-1}L_{i+1,j}L_{R,i,j}}$.

 	 Since $L_{i,j-1}L_{i+1,j}L_{R,i,j}/K$ is a separable extension, the residue field of \\$L_{i,j-1}L_{i+1,j}L_{R,i,j}$ coincides with the residue field $k$ of $K$. Hence, there exist $x_0\in k$ and $x_1\in \mathfrak{m}_{L_{i,j-1}L_{i+1,j}L_{R,i,j}}$ such that $d_{R,i,j}-x=x_0+x_1$.
	
	Since $k$ is algebraically closed by assumption, there exists $\xi_0\in k$ such that $P(\xi_0)=x_0$. Let
	\begin{equation}
		\xi_1=-\sum_{l=0}^\infty x_1^{p^{l}}.
	\end{equation}
	Then we have $\xi_1\in\mathfrak{m}_{L_{i,j-1}L_{i+1,j}L_{R,i,j}}$ and $P(\xi_1)=x_1$.
	Thus we have $P(\xi_0+\xi_1)=d_{R,i,j}-x$ and $\xi_0+\xi_1\in\mathcal{O}_{L_{i,j-1}L_{i+1,j}L_{R,i,j}}$. Hence, we obtain $\delta_{R,i,j}\equiv \xi \pmod{\mathcal{O}_{L_{i,j-1}L_{i+1,j}L_{R,i,j}}}$. Combining with \eqref{eq046}, we get 
	\begin{align*}
L_{i,j}L_{R,i,j}&=L_{R,i,j}(\delta_{R,l,m}\ (i\leq l<m\leq j))\\
&=L_{R,i,j-1}(\delta_{R,l,m}\ (i\leq l<m\leq j-1))\\ &
\quad\cdot L_{R,i+1,j}(\delta_{R,l,m}\ (i+1\leq l<m\leq j))\cdot L_{R,i,j}(\delta_{R,i,j})\\
&=L_{i,j-1}L_{R,i,j-1}\cdot L_{i+1,j}L_{R,i+1,j}\cdot L_{R,i,j}(\delta_{R,i,j})\\
&=L_{i,j-1}L_{i+1,j}L_{R,i,j}(\delta_{R,i,j})\\
&=L_{i,j-1}L_{i+1,j}L_{R,i,j}(\xi).
		\stepcounter{equation}\tag{\theequation}
	\end{align*}
Since $K_R$ is a subfield of $L_{R,i,j}$, we get \eqref{eq138}, and hence the lemma is proved.
\end{proof}

\begin{definition}\label{df006}
	We now define matrices $S_R=(s_{R,l,m})_{l,m}$, $\Sigma_R=(\sigma_{R,l,m})_{l,m}$ as follows:
	\begin{equation} \label{eq006}
		S_R=\left(W_{R,N}^{(p)}\right)^{-1}\eta_R(A)W_{R,N-1}^{(p)}=\left(W_{R,N}^{(p)}\right)^{-1}AW_{R,N-1}^{(p)}-I\in NT_n(K_R),
	\end{equation}
	\begin{equation} \label{eq045}
		P(\Sigma_R)=S_R,\ \Sigma_R\in NT_n(K_{R,sep}).
	\end{equation}
\end{definition}

We will prove that $S_R$ and $\Sigma_R$ serve as the approximations of $D_R$ and $\Delta_R$ respectively, for some appropriate $R$. Then the calculation of the ramification breaks of $L/K$ is reduced to Artin-Schreier theory by Lemma \ref{lm005}.

There exist multiple choices for matrix $\Sigma_R$ that satisfy \eqref{eq045}, but they are equivalent modulo $NT_n(\mathbb{F}_p)$ and we may arbitrarily choose one of them to be $\Sigma_R$.

Denote the $(i,j)$-th entry of $W_{R,e}^{-1}$ by $w^*_{R,e,i,j}$.

By Lemma \ref{lm001}\ref{ot004}, \ref{ot005}, \ref{ot002}, \ref{ot011}, and \ref{ot009}, we get
\begin{align}
	 & v_{K_R}(W_{R,0})=+\infty,\ v_{K_R}(W_{R,e})=-p^{e-1}m_{A}                               & (1\leq e\leq N), \\
	 & \tilde{v}_{K_R}(W_{R,0})=+\infty,\ \tilde{v}_{K_R}(W_{R,e})=p^{e-1}\tilde{v}_{K_R}(A_R) & (1\leq e\leq N), \\
	 & v_{K_R}(\Theta)=p^{-1}v_{K_R}(A)=-qp^{-1}m_{A},                                         &                  \\
	 & \tilde{v}_{K_R}(\Theta)=p^{-1}\tilde{v}_{K_R}(A)=qp^{-1}\tilde{v}_{K_R}(A_R),           &                  \\
	 & v_{K_R}(\Theta_R)=-p^{-1}m_{A}, \label{eq065}                                                              \\
	 & \tilde{v}_{K_R}(\Theta_R)=p^{-1}\tilde{v}_{K_R}(A_R).\label{eq018}
\end{align}

Note that for any $X\in FT_n(L)$, we have $v_{K_R}(X)=qv_{K_R}(\iota_R(X))$.

By the definition of $S_R$, we have $W_{R,N}^{(p)}S_R=\eta_{R}(A)W_{R,N-1}^{(p)}$.
Therefore, we get
\begin{equation}\label{eq068}
	s_{R,i,j}+\sum_{l=i+1}^{j-1} w_{R,N,i,l}^ps_{R,l,j}=\eta_R(a_{i,j})+\sum_{l=i+1}^{j-1} \eta_R(a_{i,l})w_{R,N-1,l,j}^p
\end{equation}
for any $1\leq i<j\leq n$, since $S_R$ and $\eta_{R}(A)$ are nilpotent.

\begin{lemma}\label{lm018}
	We have
	\begin{equation}\label{eq022}
		D_R-S_R=\Delta_R^{(p)}(I-(S_R+I)A_{R})+{(\Theta_R-I)}^{(p)}S_R{A_R}+S_R({A_R}-I).
	\end{equation}
\end{lemma}
\begin{proof}
	We will first show
	\begin{equation}\label{eq083}
		\Delta_R=\left(\Delta_R^{(p)}(S_R+I)-\Theta_R^{(p)}S_R\right){A_R}.
	\end{equation}
	By the definition of $\Delta_R$, we have
	\begin{equation}
		\Delta_R^{(p)}(S_R+I)-\Theta_R^{(p)}S_R=\Theta_R^{(p)}-\Theta^{(p)}W_{R,N}^{(p)}(S_R+I).
	\end{equation}
	By \eqref{eq006} and the definition of $\Theta$, we have
	\begin{equation}
		\Theta^{(p)}W_{R,N}^{(p)}(S_R+I)=\Theta^{(p)}AW_{R,N-1}^{(p)}=\Theta W_{R,N-1}^{(p)}.
	\end{equation}
	Hence we have
	\begin{equation}
		\Delta_R^{(p)}(S_R+I)-\Theta_R^{(p)}S_R=\Theta_R^{(p)}-\Theta W_{R,N-1}^{(p)}.
	\end{equation}
	By \eqref{eq082} and \eqref{eq062} and the definition of $\Delta_R$, we have
	\begin{equation}
		\left(\Delta_R^{(p)}(S_R+I)-\Theta_R^{(p)}S_R\right)A_R=\Theta_R-\Theta W_{R,N}=\Delta_R,
	\end{equation}
	which is precisely \eqref{eq083}.

	By \eqref{eq083}, we have
	\begin{equation}
		D_R=\Delta_R^{(p)}-\Delta_R=\Delta_R^{(p)}(I-(S_R+I)A_{R})+\Theta_R^{(p)}S_RA_R.
	\end{equation}
	Since
	\begin{equation}
		\Theta_R^{(p)}S_RA_R-S_R={(\Theta_R-I)}^{(p)}S_R{A_R}+S_R({A_R}-I),
	\end{equation}
	we obtain \eqref{eq022}.
\end{proof}

\begin{proposition}\label{pp001}
	Let $R_0=q\max(r_{1,n-1},r_{2,n})+(n-1)m_{A}$ and $R> R_0$ be an integer prime to $p$. Let $c=R-q\max(r_{1,n-1},r_{2,n})$. Suppose that the following conditions hold.
	\begin{enumerate}[label=(\alph*)]

		\item\label{ot013} $\tilde{v}_{K_R}(t_R^{c}S_R) >-m_{A}$
		\item\label{ot014} $\tilde{v}_{K_R}(t_R^{c}D_R)>-m_{A}$

	\end{enumerate}
	Then we have

	$v_{K_{R}}(d_{R,1,n}-s_{R,1,n})>c-(n-1)m_{A}>0$ and $v_{K_R}(D_R-S_R)>0$.

\end{proposition}

\begin{proof}

	We will prove this proposition by evaluating the valuations of the three terms of the right-hand side of \eqref{eq022}.

	Applying Lemma \ref{lm001}\ref{ot010} to condition \ref{ot013} and \ref{ot014}, we get
	\begin{equation}\label{eq063}
		\tilde{v}_{K_R}(S_R) \geq \left(\frac{c}{n-2}-m_{A}\right)>0
	\end{equation}
	and
	\begin{equation}
		\tilde{v}_{K_R}(D_R) \geq \left(\frac{c}{n-2}-m_{A}\right)>0.
	\end{equation}

	Hence, we obtain
	\begin{equation}
		v_{K_R}(d_{R,i,j})=v_{K_R}(\delta_{R,i,j})>0
	\end{equation}
	for all $1\leq i<j\leq n$ except for $(i,j)=(1,n)$. Therefore, we have
	\begin{equation}
		\tilde{v}_{K_R}(D_R)=\tilde{v}_{K_R}(\Delta_R)
	\end{equation}
	and
	\begin{equation}\label{eq064}
		\tilde{v}_{K_R}(t_R^{c}D_R)=\tilde{v}_{K_R}(t_R^c\Delta_R)>-m_{A}
	\end{equation}
	by condition \ref{ot014}.

	It follows by \eqref{eq063} and Lemma \ref{lm001}\ref{ot007} and \ref{ot004} that

	\begin{equation}\label{eq016}
		\tilde{v}_{K_R}((S_R+I)A_{R}) \geq\min(\tilde{v}_{K_R}(S_R),\tilde{v}_{K_R}({A_R})) =\tilde{v}_{K_R}({A_R})\geq -m_A.
	\end{equation}

	Therefore, by Lemma \ref{lm001}\ref{ot010}, we obtain
	\begin{align*}\label{eq139}
		& v_{K_R}(t_R^{c}\Delta_R^{(p)}(I-(S_R+I)A_{R}))\\
		& =v_{K_{R}}\left(t_{R}^{-(p-1)c}(t_R^{c}\Delta_R)^{(p)}(I-(S_R+I)A_{R})\right) \\
		                                               & \geq \frac{(p-1)c}{n-1}+v_{K_{R}}\left((t_R^{c}\Delta_R)^{(p)}(I-(S_R+I)A_{R})\right).
		\stepcounter{equation}\tag{\theequation}
	\end{align*}
	By Lemma \ref{lm001}\ref{ot004}, we deduce that
	\begin{align*}\label{eq140}
		& v_{K_{R}}\left((t_R^{c}\Delta_R)^{(p)}(I-(S_R+I)A_{R})\right)\\
		&\geq \min\left(v_{K_{R}}\left((t_R^{c}\Delta_R)^{(p)}\right),v_{K_R}(I-(S_R+I)A_{R})\right).
			      \stepcounter{equation}\tag{\theequation}
		      \end{align*}
	Applying Lemma \ref{lm001}\ref{ot007} and \ref{ot008}, we get
	\begin{equation}
		v_{K_{R}}\left((t_R^{c}\Delta_R)^{(p)}\right)=pv_{K_{R}}(t_R^{c}\Delta_R)
	\end{equation}
	and
	\begin{equation}
		v_{K_R}(I-(S_R+I)A_{R})=v_{K_R}((S_R+I)A_{R})
	\end{equation}
	By \eqref{eq064} and \eqref{eq016}, it follows that
	\begin{align*}\label{eq141}
		\min\left(v_{K_{R}}\left((t_R^{c}\Delta_R)^{(p)}\right),v_{K_R}(I-(S_R+I)A_{R})\right)&		\geq \min(-pm_A,-m_{A})\\
		&=-pm_{A}.
			      \stepcounter{equation}\tag{\theequation}
		      \end{align*}
	Combining \eqref{eq139}, \eqref{eq140}, and \eqref{eq141}, we get
	\begin{equation}
		v_{K_R}(t_R^{c}\Delta_R^{(p)}(I-(S_R+I)A_{R}))\geq\frac{(p-1)c}{n-1}-pm_{A}.
	\end{equation}
	Since $c>(n-1)m_{A}$ by definition, we have \begin{equation}
		\frac{(p-1)c}{n-1}-pm_{A}>((p-1)-p)m_A= -m_A.
	\end{equation}
	Hence we see that
	\begin{equation}\label{eq019}
		v_{K_R}(t_R^{c}\Delta_R^{(p)}(I-(S_R+I)A_{R}))>-m_{A}.
	\end{equation}

	By Lemma \ref{lm001}\ref{ot005} and \ref{ot004}, we have
	\begin{align*}
		& v_{K_R}(t_R^{c}(\Theta_R-I)^{(p)}S_R{A_R})\\
		 & \geq\min\left(v_{K_R}\left((\Theta_R-I)^{(p)}\right),\tilde{v}_{K_R}(t_R^{c}S_R{A_R})\right)                   \\
		                                           & \geq\min\left(v_{K_R}\left((\Theta_R-I)^{(p)}\right),\tilde{v}_{K_R}(t_R^{c}S_R),\tilde{v}_{K_R}({A_R})\right).
		\stepcounter{equation}\tag{\theequation}
	\end{align*}
	By condition \ref{ot013}, and Lemma \ref{lm001}\ref{ot003}, \ref{ot007}, and \ref{ot008}, we get
	\begin{align*}
		& \min\left(v_{K_R}\left((\Theta_R-I)^{(p)}\right),\tilde{v}_{K_R}(t_R^{c}S_R),\tilde{v}_{K_R}({A_R})\right)\\
		 & \geq\min\left(pv_{K_R}(\Theta_R),-m_A,v_{K_R}(A_R)\right) \\
		                                                                                                           & \geq\min\left(pv_{K_R}(\Theta_R),-m_A\right).
		\stepcounter{equation}\tag{\theequation}
	\end{align*}
	By \eqref{eq065}, we obtain
	\begin{equation}
		\min\left(pv_{K_R}(\Theta_R),-m_A\right)=-m_A.
	\end{equation}
	Hence, we get
	\begin{equation}\label{eq021}
		v_{K_R}(t_R^{c}(\Theta_R-I)^{(p)}S_R{A_R})\geq -m_A.
	\end{equation}

	By condition \ref{ot013}, Lemma \ref{lm001}\ref{ot007}, and \ref{ot006}, we get
	\begin{align*}
		v_{K_R}(t_R^{c}S_R(I-{A_R})) & \geq \min(\tilde{v}_{K_R}(t_R^{c}S_R),\tilde{v}_{K_R}(I-{A_R}^{-1})) \\
		                             & = v_{K_R}({A_R})                                                     \\
		                             & =-m_A.
		\label{eq026}
		\stepcounter{equation}\tag{\theequation}
	\end{align*}

	Therefore, by \eqref{eq022}, \eqref{eq019}, \eqref{eq021}, \eqref{eq026}, and Lemma \ref{lm001}\ref{ot007}, we get
	\begin{equation}
		v_{K_R}(t_R^c(D_R-S_R))\geq -m_A.
	\end{equation}
	By Lemma \ref{lm001}\ref{ot010}, we get
	\begin{equation}
		v_{K_R}(D_R-S_R)\geq \frac{c}{n-1}-m_A>0.
	\end{equation}
	Consequently, we have
	\begin{equation}
		v_{K_R}(d_{R,1,n}-s_{R,1,n})\geq c-(n-1)m_A>0.
	\end{equation}
\end{proof}

Define $m_{A,i,j}$ and $R_{i,j}$ as
\begin{equation}\label{eq084}
	m_{A,i,j}=-\min_{i\leq l<m\leq j}\left(\frac{v_{K}(a_{l,m})}{m-l}\right)
\end{equation} and
\begin{equation}\label{eq085}
	R_{i,j}=q\max(r_{i,j-1},r_{i+1,j})+(j-i)m_{A,i,j}
\end{equation}
for $1\leq i<j\leq n$.
Note that $-m_{A,i,j}$ is the valuation of the $(j-i+1)\times(j-i+1)$ submatrix of $A$ consisting of $i$-th to $j$-th columns and $i$-th to $j$-th rows.
\begin{lemma}\label{lm012}
	Let $1\leq i<j\leq n$. Let  $R>R_{i,j}$ be an integer prime to $p$. Suppose we have $v_{K_R}(s_{R,i,j})=v_{K_{R}}(d_{R,l,m})=-qr_{l,m}+R$ for $i\leq l<m\leq j$ except for $(l,m)=(i,j)$. Then we have $v_{K_{R}}(d_{R,i,j}-s_{R,i,j})>R-R_{i,j}>-qr_{i,j}+R$.
\end{lemma}
\begin{proof}

	Let $c= R-q\max(r_{i,j-1},r_{i+1,j})$.
	Note that by Corollary \ref{cr002}, we have $R_{i,j}>R_{l,m}$ for $i\leq l<m\leq j$ such that $j-i>m-l$. By the assumption that $v_{K_R}(s_{R,l,m})=v_{K_R}(d_{R,l,m})=-qr_{l,m}+R$ for all $i\leq l<m\leq j$ except for $(l,m)=(i,j)$ and Corollary \ref{cr002}, we get
	\begin{equation}\label{eq142}
		v_{K_{R}}(t_{R}^{c}s_{R,l,m})=v_{K_{R}}(t_{R}^{c}d_{R,l,m})=-qr_{l,m}+R-c\geq 0>-(m-l)m_{A,i,j}
	\end{equation}
	for all $i\leq l<m\leq j$ except for $(l,m)=(i,j)$. We apply Proposition \ref{pp001} to the $(j-i+1)\times(j-i+1)$ submatrix of $A$ consisting of $i$-th to $j$-th columns and $i$-th to $j$-th rows, setting $R$ in the proposition to be our current $R$. Note that $R_0$ in the proposition equals $R_{i,j}$ in this setting, and the conditions of the proposition follow from \eqref{eq142}. We get
	\begin{equation}
		v_{K_{R}}(d_{R,i,j}-s_{R,i,j})>c-(j-i)m_{A,i,j}=R-R_{i,j}.
	\end{equation}
	It follows from Lemma \ref{lm008} that $qr_{i,j}>q\max(r_{i,j-1},r_{i+1,j})+nm_A+p$. Since we have $m_A\geq m_{A,i,j}$ by definition, we get 
	\begin{align*}
		R_{i,j}&=q\max(r_{i,j-1},r_{i+1,j})+(j-i)m_{A,i,j}\\
		&<q\max(r_{i,j-1},r_{i+1,j})+nm_A+p\\
		&<qr_{i,j},
			      \stepcounter{equation}\tag{\theequation}
		      \end{align*}
	and we obtain $R-R_{i,j}>-qr_{i,j}+R$.
\end{proof}

\section{Calculation of the Valuations of the Matrices and the Ramification Breaks}\label{sc005}

Let $K$ be a complete discrete valuation field of equal characteristic $p>0$ with algebraically closed residue field $k$ and $L$ be a totally wildly ramified finite Galois extension of $K$ such that $\Gal(L/K)\simeq UT_n(\mathbb{F}_p)$ for some $n\geq 2$.
Fix an integer $N$ and $q=p^N$ as in Section \ref{sc004}, and define $K_R$ as in Definition \ref{df003} for any positive integer $R$ prime to $p$. Define $\iota_R$ as in Definition \ref{df004}. Fix $A$ satisfying the conditions of Lemma \ref{lm011} and define $A_R$, $W_{R,i}$, $D_R$, $S_R$ as in Definition \ref{df005} and Definition \ref{df006}. Define $L_{i,j}$, $r_{i,j}$, and $R_{i,j}$ as in Section $\ref{sc004}$. Define $\coef_{K,l}:K\to k$ as in Definition \ref{df002} for any $l\in \mathbb{Z}$.

In this section, we will relate the ramification breaks with the valuations of the entries of $S_R$. By the definition of $S_R$, we can express $s_{R,i,j}$ for $1\leq i<j\leq n$ in terms of the coefficients of the defining equation of $L/K$.

Recall the definition of the partitions in Definition \ref{df007}.

\begin{definition}\label{df008}
	Define $m_{i,j}$, $m_{\lambda}$, $\mu_{i,j}$ for $1\leq i<j\leq n$ and $\lambda\in\Lambda_{i,j}$ as follows:
	\begin{align}
		 & m_{i,j}=-v_{K}(a_{i,j}),                                           \\
		 & m_{\lambda}=\sum_{u=0}^{\lvert\lambda\rvert -1}m_{\lambda_{u},\lambda_{u+1}}, \\
		 & \mu_{i,j}=\max_{\lambda\in\Lambda_{i,j}}m_{\lambda}.
	\end{align}
\end{definition}

Note that we have \begin{equation}\label{eq122}
	v_{K}(a_{i,j}^*)\geq -\mu_{i,j}
\end{equation}
by Lemma \ref{lm027}\ref{ot058} for $\nu_{i,j}=\mu_{i,j}$ for all $1\leq i<j\leq n$.

By definition, we have
\begin{equation}\label{eq124}
	\mu_{i,j}\geq \mu_{i,l}+\mu_{l,j}
\end{equation}
for any $1\leq i<l<j\leq n$ and
\begin{equation}\label{eq136}
	\mu_{i,j}\geq m_{i,j}
\end{equation}
for any $1\leq i<j\leq n$.
By the definition of $A$, we have $m_{i,i}=0$ for all $1\leq i\leq n$, and $p\nmid m_{i,j}>0$ or $m_{i,j}=-\infty$ for all $1\leq i<j\leq n$.
\begin{lemma}\label{lm014}
	We have $p\nmid m_{i,i+1}=r_{i,i+1}>0$ for all $1\leq i \leq n-1$.
\end{lemma}
\begin{proof}
	Since $\Gal(L/K)\simeq UT_n(\mathbb{F}_p)$, we have $[L:K]=p^{\frac{n(n-1)}{2}}$. On the other hand, $L_{i,j}/L_{i,j-1}L_{i+1,j}$ is an Artin-Schreier extension for all $1\leq i<j\leq n$. Comparing the degrees of the extensions, we see that $L_{i,j}/L_{i,j-1}L_{i+1,j}$ is of degree $p$ for all $1\leq i<j\leq n$. Therefore, $K(\theta_{i,i+1})/K$ is not trivial, and consequently, $a_{i,i+1}=-P(\theta_{i,i+1})\neq 0$.
	By the definition of $A$, we have $p\nmid m_{i,i+1}=-v_{K}(a_{i,i+1})>0$. Hence, we obtain $r_{i,i+1}=m_{i,i+1}$.
\end{proof}

\begin{lemma}\label{lm015}
	Let $1\leq e\leq N$ and $1\leq i<j\leq n$. We have
	\begin{equation}\label{eq116}
		w_{R,e,i,j}\equiv a_{R,i,j}^{p^{e-1}}\pmod{\mathfrak{m}_{K_R}^{-p^{e-1}\mu_{i,j}+1}},
	\end{equation}
	and
	\begin{equation}\label{eq102}
		v_{K_R}(w_{R,e,i,j})\geq -p^{e-1}\mu_{i,j}.
	\end{equation}
	The equality in \eqref{eq102} holds if and only if $m_{i,j}=\mu_{i,j}$. We also have
	\begin{equation}\label{eq117}
		w_{R,e,i,j}^*\equiv a_{R,i,j}^{*p^{e-1}}\pmod{\mathfrak{m}_{K_R}^{-p^{e-1}\mu_{i,j}+1}}
	\end{equation}
	and
	\begin{equation}\label{eq103}
		v_{K_R}(w^*_{R,e,i,j})\geq -p^{e-1}\mu_{i,j}.
	\end{equation}
	The equality in \eqref{eq103} holds if and only if the equality in \eqref{eq122} holds for the pair $(i,j)$.
\end{lemma}

\begin{proof}
	This lemma follows from the definition of $W_{R,e}$ and Lemma \ref{lm027}\ref{ot058} and \ref{ot061} for $\nu_{i,j}=p^{e-1}\mu_{i,j}$ for all $1\leq i<j\leq n$.

\end{proof}

For all $1\leq i<j\leq n$, we have
\begin{equation}\label{eq034}
	\eta_R(a_{i,j})=\sum_{l=-m_{i,j}}^{\infty}\coef_{K,l}(a_{i,j})t_R^{-ql}\sum_{m=1}^{\infty}\binom{-\frac{l}{R}}{m}t_R^{-mR}
\end{equation}
by \eqref{eq032}. Thus we have
\begin{equation}\label{eq035}
	v_{K_R}(\eta_R(a_{i,j}))=-qm_{i,j}+R.
\end{equation}
\begin{definition}\label{df009}
	Define $a_{i,j}^{[m]}$ and $s_{R,i,j}^{[m]}$ for $1\leq i<j\leq n$ and $m\in\mathbb{Z}_{\geq 0}$ as follows:
	\begin{equation}
		a_{i,j}^{[m]}=\sum_{l=-m_{i,j}}^{\infty}\coef_{K,l}(a_{i,j})\binom{-\frac{l}{R}}{m}t_R^{-ql-mR},
	\end{equation}
	\begin{equation}\label{eq061}
		s_{R,i,j}^{[m]}=\sum_{i\leq i'<j'\leq j}w_{R,N,i,i'}^{*p}a_{i',j'}^{[m]}w^{p}_{R,N-1,j',j}.
	\end{equation}
	Define $a'_{i,j}$ for $1\leq i<j\leq n$ as follows:
	\begin{equation}
		a'_{i,j}=\sum_{l=-m_{i,j}}^{\infty}l\coef_{K,l}(a_{i,j})t^{-l}
	\end{equation}
\end{definition}

We define $\coef_{K_R,i}$ for $K_R$ in the same manner as Definition \ref{df002}.
\begin{lemma}\label{lm016}
	For any $1\leq i<j\leq n$ and $m\in\mathbb{Z}_{\geq 0}$, $a_{i,j}^{[m]}$, $s_{R,i,j}^{[m]}$, and $a'_{i,j}$ satisfy the following properties:
	\begin{enumerate}[label=(\roman*)]
		\item\label{ot025}
		\begin{equation}
			a_{i,j}=\sum_{m=0}^{\infty}a_{i,j}^{[m]}.
		\end{equation}
		\item\label{ot026}
		\begin{equation}
			a_{i,j}^{[0]}=a_{R,i,j}^{q}.
		\end{equation}
		\item\label{ot027}
		\begin{equation}
			\eta_R(a_{i,j})=\sum_{m=1}^{\infty}a_{i,j}^{[m]}.
		\end{equation}
		\item\label{ot028}
		\begin{equation}\label{eq144}
			s_{R,i,j}=\sum_{m=1}^{\infty}s_{R,i,j}^{[m]}.
		\end{equation}
		\item\label{ot054}
		\begin{equation}\label{eq126}
			a_{i,j}^{[1]}=-\frac{t_R^{-R}}{R}\iota_R(a'_{i,j})^q
		\end{equation}
		\item\label{ot043}
		\begin{equation}\label{eq058}
			\coef_{K_{R},l}(a_{i,j}^{[m]})\neq 0 \Rightarrow l\equiv mR \pmod{q}.
		\end{equation}
		\item\label{ot022}
		\begin{equation}\label{eq038}
			v_{K_R}\left(a_{i,j}^{[m]}\right)=-qm_{i,j}+mR.
		\end{equation}
		\item\label{ot053}
		\begin{equation}\label{eq125}
			v_{K}\left(a'_{i,j}\right)=-m_{i,j}.
		\end{equation}
		\item\label{ot051}
		\begin{equation}\label{eq123}
			v_{K_R}\left(\sum_{l=i}^{j-1}a_{i,l}^*a_{l,j}^{[m]}\right)\geq -q\mu_{i,j}+mR.
		\end{equation}
		\item\label{ot052}
		\begin{equation}\label{eq115}
			v_{K}\left(\sum_{l=i}^{j-1}a_{i,l}^*a'_{l,j}\right)\geq -\mu_{i,j}.
		\end{equation}
		\item\label{ot023}
		\begin{equation}\label{eq039}
			v_{K_R}\left(s_{R,i,j}^{[m]}\right)\geq -q\mu_{i,j}+mR.
		\end{equation}
		The equality in \eqref{eq039} holds if and only if the equality in \eqref{eq123} holds.
		\item\label{ot024}
		\begin{equation}\label{eq037}
			v_{K_R}\left(s_{R,i,j}^{[m]}\right)\equiv mR \pmod{p}.
		\end{equation}
		\item\label{ot033}
		\begin{equation}
			v_{K_R}\left(\eta_R(a_{i,j})\right)=-qm_{i,j}+R.
		\end{equation}
		\item\label{ot032}
		\begin{equation}\label{eq127}
			v_{K_R}\left(s_{R,i,j}\right)\geq -q\mu_{i,j}+R.
		\end{equation}
		The equality in \eqref{eq127} holds if and only if the equality in \eqref{eq115} holds.
		\item\label{ot039}
		\begin{equation}
			v_{K_R}\left(s_{R,i,i+1}\right)= -qm_{i,i+1}+R.
		\end{equation}
	\end{enumerate}
\end{lemma}
\begin{proof}
	\ref{ot025}--\ref{ot053} follow from \eqref{eq032}, \eqref{eq034}, and the definition of $a_{i,j}^{[m]}$, $s_{R,i,j}^{[m]}$, and $a'_{i,j}$.
	\begin{enumerate}[label=(\roman*)]
		\addtocounter{enumi}{8}
		\item This follows from \ref{ot022}, \eqref{eq122}, \eqref{eq124} and \eqref{eq136}.
		\item This follows from \ref{ot053}, \eqref{eq122}, \eqref{eq124} and \eqref{eq136}.
		\item
		      Combining Lemma \ref{lm015} and \ref{ot022} with \eqref{eq061}, we get the inequality.

		      Since
		      \begin{align*}
			      v_{K_R}\left(w_{R,N,i,i'}^{*p}a_{i',j'}^{[m]}w^{p}_{R,N-1,j',j}\right)&\geq -q\left(\mu_{i,i'}+m_{i',j'}+\frac{\mu_{j',j}}{p}\right)+mR\\
						&>-q\mu_{i,j}+mR
			      \stepcounter{equation}\tag{\theequation}
		      \end{align*}
		      for all $i\leq i'<j'<j$, we obtain
		      \begin{equation}
			      s_{R,i,j}^{[m]}\equiv \sum_{i'=i}^{j-1}w_{R,N,i,i'}^{*p}a_{i',j}^{[m]} \pmod{\mathfrak{m}_{K_R}^{-q\mu_{i,j}+mR+1}}.
		      \end{equation}
		      By \eqref{eq117} and \ref{ot022}, we get
		      \begin{equation}
			      w_{R,N,i,i'}^{*p}a_{i',j}^{[m]}\equiv a_{R,i,i'}^{*q}a_{i',j}^{[m]} \pmod{\mathfrak{m}_{K_R}^{-q\mu_{i,j}+mR+1}},
		      \end{equation}
		      for all $i\leq i'<j'<j$, since $\mu_{i,i'}+m_{i',j}\leq \mu_{i,j}$. We have
		      \begin{align*}
			      \coef_{K_R,-qm_{i,i'}}\left(a_{R,i,i'}^{*q}\right)&=\coef_{K_R,-qm_{i,i'}}\left(\iota_{R}\left(a_{i,i'}^{*}\right)^q\right)\\
						&=\coef_{K_R,-qm_{i,i'}}\left(a_{i,i'}^{*}\right)
			      \stepcounter{equation}\tag{\theequation}
		      \end{align*}
		      by the definition of $\iota_R$ and $t_R$. Hence, we get the assertion.

		\item By \ref{ot043}, we obtain
		      \begin{equation}
			      \coef_{K_{R},l}(x^pa_{i,j}^{[m]})\neq 0 \Rightarrow l\equiv mR \pmod{p}
		      \end{equation}
		      for any $x\in K_{R}$. Thus, by \eqref{eq061}, we have
		      \begin{equation}
			      \coef_{K_{R},l}(s_{R,i,j}^{[m]})\neq 0 \Rightarrow l\equiv mR \pmod{p}.
		      \end{equation}
		      This gives the assertion.
		\item By \ref{ot027} and \ref{ot022}, we get the assertion.
		\item By \ref{ot028} and \ref{ot023}, we get the inequality.
		Furthermore, we deduce that the equality holds if and only if 
		\begin{equation}\label{eq143}
			v_{K_R}\left(s_{R,i,j}\right)=v_{K_R}\left(s_{R,i,j}^{[1]}\right)= -q\mu_{i,j}+R.
		\end{equation}
		The first equality of \eqref{eq143} holds if the second equality holds, since $s_{R,i,j}^{[1]}$ is the only term in the right-hand side of \eqref{eq144} whose valuation in $K$ can be equal to $-q\mu_{i,j}+R$.
		Thus, by \ref{ot023},  \eqref{eq143} holds if and only if the equality in \eqref{eq123} for $m=1$ holds. By \ref{ot054}, it follows that
		\begin{equation}
			\sum_{l=i}^{j-1}a_{i,l}^*a_{l,j}^{[1]}=-\frac{t_R^{-R}}{R}\left(\sum_{l=i}^{j-1}a_{i,l}^*a'_{l,j}-\sum_{l=i}^{j-1}a_{i,l}^*\eta_R(a'_{l,j})\right).
		\end{equation}
		By the same argument as \eqref{eq035}, we get $v_{K_R}(\eta_R(a'_{l,j}))=-qm_{i,j}+R$. Therefore, we have
		\begin{equation}
			v_{K_R}\left(\frac{t_R^{-R}}{R}\sum_{l=i}^{j-1}a_{i,l}^*\eta_R(a'_{l,j})\right)\geq -q\mu_{i,j}+2R.
		\end{equation}
		This implies that the equality in \eqref{eq123} for $m=1$ holds if and only if 
		\begin{equation}\label{eq145}
			v_{K_R}\left(-\frac{t_R^{-R}}{R}\sum_{l=i}^{j-1}a_{i,l}^*a'_{l,j}\right)=-q\mu_{i,j}+R.
		\end{equation}
		The equation \eqref{eq145} is equivalent to the equality in \eqref{eq115}.

		\item By \eqref{eq068}, we get $s_{R,i,i+1}=\eta_R(a_{i,i+1})$. We get the assertion by \ref{ot027} and \ref{ot022}.
	\end{enumerate}
\end{proof}

\begin{remark}\label{rm002}
	Let $d:K\to \Omega_{K/k}^1$ be the canonical derivation of the K\"ahler differential. Then we have
	\begin{equation}
		da_{i,j}=\sum_{l=-m_{i,j}}^{\infty}-l\coef_{K,l}(a_{i,j})t^{-l-1}dt=-a'_{i,j}\frac{dt}{t}
	\end{equation}

	Thus the element $a'_{i,j}\in K$ corresponds to $da_{i,j}$ in \cite{Im}.
	We will see that the results of this paper are consistent with the results of \cite{Im} in Example \ref{ex001}.
\end{remark}

\begin{lemma}\label{lm023}
	Let $1\leq i<j\leq n$ and suppose $j-i>1$ and the inequality
	\begin{equation}\label{eq079}
		\mu_{i,j}=m_{i,j}>\mu_{i,l}+\mu_{l,j}
	\end{equation}
	holds for all $i<l<j$. Then we have \begin{equation}\label{eq071}
		v_{K_R}(s_{R,i,j})=v_{K_R}(\eta_{R}(a_{i,j}))=-qm_{i,j}+R.
	\end{equation}
\end{lemma}
\begin{proof}
	By Lemma \ref{lm015} and Lemma \ref{lm016}\ref{ot033} and \ref{ot032}, we get
	\begin{equation}\label{eq069}
		v_{K_R}\left(\sum_{l=i+1}^{j-1} \left(\eta_R(a_{i,l})w_{R,N-1,l,j}^p-w_{R,N,i,l}^ps_{R,l,j}\right)\right)\geq -q\max_{i<l<j}\left(\mu_{i,l}+\mu_{l,j}\right)+R
	\end{equation}
	Combining \eqref{eq068}, \eqref{eq079}, \eqref{eq069}, and Lemma \ref{lm016}\ref{ot033}, we obtain \eqref{eq071}.
\end{proof}

\begin{lemma}\label{lm021}
	Let $1\leq i<j\leq n$, and suppose $\lambda\in\Lambda_{i,j}$ is one of the longest partitions satisfying
	\begin{equation}\label{eq070}
		\mu_{i,j}=m_{\lambda}.
	\end{equation}
	Then we have
	\begin{equation}\label{eq072}
		\mu_{i,j}=-\sum_{u=0}^{\lvert\lambda\rvert -1}\frac{v_{K_R}(s_{R,\lambda_{u},\lambda_{u+1}})-R}{q}
	\end{equation}
\end{lemma}
\begin{proof}
	We have
	\begin{equation}\label{eq067}
		\mu_{\lambda_{u},\lambda_{u+1}}=m_{\lambda_{u},\lambda_{u+1}}>\mu_{\lambda_{u},m}+\mu_{m,\lambda_{u+1}}
	\end{equation}
	for all $0\leq u\leq \lvert\lambda\rvert -1$ and $\lambda_{u}<m<\lambda_{u+1}$. Indeed, if $\mu_{\lambda_{u},\lambda_{u+1}}=\mu_{\lambda_{u},m}+\mu_{m,\lambda_{u+1}}$ for some $\lambda_{u}<m<\lambda_{u+1}$, then there exist a partition $\lambda'\in\Lambda_{\lambda_{u},\lambda_{u+1}}$ satisfying
	\begin{equation}
		m_{\lambda_{u},\lambda_{u+1}}=\mu_{\lambda_{u},\lambda_{u+1}}=\mu_{\lambda_{u},m}+\mu_{m,\lambda_{u+1}}=m_{\lambda'}.
	\end{equation}
	Then we can make a longer partition satisfying \eqref{eq070} than $\lambda$ by substituting the part $\lambda_{u}<\lambda_{u+1}$ of $\lambda$ by $\lambda'$. This contradicts the maximality of the length of $\lambda$, and therefore we get the assertion by Lemma \ref{lm023} and \eqref{eq067}.
\end{proof}

Lemma \ref{lm021} shows that there exists at least one partition satisfying \eqref{eq072}.
\begin{lemma}\label{lm024}
	Let $1\leq i<j \leq n$ and $R>R_{i,j}$ be an integer prime to $p$. Let $\lambda$ be the longest partition in $\Lambda_{i,j}$, i.e. $\lambda$ satisfies $\lvert\lambda\rvert =j-i$ and $\lambda_{u}=i+u$ for $0\leq u\leq j-i$. Suppose that $\lambda$ is the only partition for $(i,j)$ satisfying \eqref{eq072} in Lemma \ref{lm021}. Then we have
	\begin{equation}
		\mu_{l,m}=m_{l,m}=\sum_{l'=l}^{m-1}m_{l',l'+1}
	\end{equation}
	for all $i\leq l<m\leq j$.
\end{lemma}
\begin{proof}

	By Lemma \ref{lm016}\ref{ot039}, we get
	\begin{equation}\label{eq087}
		\mu_{i,j}=\sum_{l=i}^{j-1}m_{l,l+1}.
	\end{equation}
	For any $i\leq l<m\leq j$, we have
	\begin{equation}\label{eq086}
		\mu_{i,j}\geq \mu_{i,l}+\mu_{l,m}+\mu_{m,j}\geq\sum_{l'=i}^{l-1}m_{l',l'+1}+\sum_{l'=l}^{m-1}m_{l',l'+1}+\sum_{l'=m}^{j-1}m_{l',l'+1}=\sum_{l'=i}^{j-1}m_{l',l'+1}
	\end{equation}
	by definition. By \eqref{eq087}, we get that the inequalities in \eqref{eq086} are equalities. Thus, by Lemma \ref{lm016}\ref{ot039}, we have
	\begin{equation}\label{eq089}
		-q\mu_{l,m}=-q\sum_{l'=l}^{m-1}m_{l',l'+1} =\sum_{l'=l}^{m-1}(v_{K_R}(s_{R,l',l'+1})-R)
	\end{equation}
	for all $i\leq l<m\leq j$.
	By Lemma \ref{lm016}\ref{ot032} and \eqref{eq089}, we get
	\begin{equation} \label{eq088}
		v_{K_R}(s_{R,l,m})-R\geq \sum_{l'=l}^{m-1}(v_{K_R}(s_{R,l',l'+1})-R)
	\end{equation}
	for all $i\leq l<m\leq j$.
	If the equality holds in \eqref{eq088} for some $i\leq l<m\leq j$ with $m-l>1$, then we have
	\begin{equation}
		-q\mu_{i,j}=\sum_{l'=i}^{l-1}(v_{K_R}(s_{R,l',l'+1})-R)+(v_{K_R}(s_{R,l,m})-R)+\sum_{l'=m}^{j-1}(v_{K_R}(s_{R,l',l'+1})-R).
	\end{equation}
	This implies that there exists a partition shorter than $\lambda$ satisfying \eqref{eq072}, which contradicts the assumption. Therefore the equality does not hold in \eqref{eq088} for all $i\leq l<m\leq j$ with $m-l>1$. Hence, combining with \eqref{eq089}, we obtain that the equality in 
	\begin{equation}\label{eq146}
		v_{K_R}\left(s_{R,l,m}\right)\geq -q\mu_{l,m}+R.
	\end{equation}
	 holds if and only if $m-l=1$.

	We will now show that $m_{l,m}=\mu_{l,m}$ for $i\leq l<m\leq j$ by induction on $m-l$.
	\begin{enumerate}[label=\textlangle \arabic*\textrangle]
		\item Case $m-l=1$

		      This case is immediate from the definition of $\mu_{l,m}$.
		\item Case $m-l>1$

		      We will evaluate the valuation of each term in \eqref{eq068}.
		      By the induction hypothesis, we have $m_{l,l'}=\mu_{l,l'}$ and $m_{l',m}=\mu_{l',m}$  for all $l<l'<m$. By Lemma \ref{lm015}, we deduce that
					 \begin{equation}\label{eq147}
						v_{K_R}(w^p_{R,N,l,l'})=-q\mu_{l,l'}
					\end{equation}
					 and 
					 \begin{equation}
						v_{K_R}(w^p_{R,N-1,l',m})=-qp^{-1}\mu_{l',m}.
					\end{equation} 
		      Thus we have
		      \begin{equation}\label{eq080}
			      v_{K_{R}}(\eta_{R}(a_{l,l'})w^p_{R,N-1,l',m})=-q\left(m_{l,l'}+\frac{\mu_{l,m}}{p}\right)+R>-q\mu_{l,m}+R
		      \end{equation}
		      for all $i\leq l<l'<m\leq j$, and
		      \begin{equation}\label{eq081}
			      v_{K_{R}}(w^p_{R,N,l,l'}s_{R,l',m})\geq -q\left(\mu_{l,l'}+\mu_{l',m}\right)+R=-q\mu_{l,m}+R
		      \end{equation}
		      for all $i\leq l<l'<m\leq j$ by \eqref{eq146} and \eqref{eq147}. Since the equality in \eqref{eq146} holds if and only if $l=m-1$, the equality in \eqref{eq081} also holds if and only if $l'=m-1$. 

		      We now evaluate the valuation of each term of \eqref{eq068} for $(i,j)=(l,m)$ by \eqref{eq146}, \eqref{eq080}, and \eqref{eq081}. Taking in account that the equality in \eqref{eq146} does not hold for $l<m-1$ and that the equality in \eqref{eq081} does not hold for $l'<m-1$, we deduce that $v_{K_{R}}(w^p_{R,N,l,m-1}s_{R,m-1,m}-\eta_{R}(a_{l,m}))>-q\mu_{l,m}+R$.
		      Hence, we must have
		      \begin{equation}
			      v_{K_{R}}(w^p_{R,N,l,m-1}s_{R,m-1,m})=v_{K_{R}}(\eta_{R}(a_{l,m})),
		      \end{equation} since the equality holds in \eqref{eq081} if $l'=m-1$.
		      This gives $-q\mu_{l,m}+R=-qm_{l,m}+R$ by Lemma \ref{lm016}\ref{ot033}, and consequently, $\mu_{l,m}=m_{l,m}$. Combining with \eqref{eq089}, we obtain the assertion.
	\end{enumerate}
\end{proof}

\begin{lemma}\label{lm009}
	If $v_{K_{R}}(s_{R,i,j})< -q\mu_{i,j}+pR$, then we have $p\nmid v_{K_{R}}(s_{R,i,j})$.
\end{lemma}
\begin{proof}
	Since
	\begin{equation}
		v_{K_{R}}(s_{R,i,j}^{[m]})\geq -q\mu_{i,j}+mR\geq-q\mu_{i,j}+pR>v_{K_{R}}(s_{R,i,j})
	\end{equation}
	for $m\geq p$, we have
	\begin{equation}\label{eq040}
		v_{K_{R}}(s_{R,i,j})\geq \min_{1\leq m<p}(v_{K_{R}}(s_{R,i,j}^{[m]})).
	\end{equation}
	By \eqref{eq037}, $v_{K_{R}}(s_{R,i,j}^{[m]})\ (1\leq m<p)$ are different from each other.
	Therefore the equality holds in \eqref{eq040}, and we get $p\nmid v_{K_{R}}(s_{R,i,j})$.
\end{proof}

\begin{lemma}\label{lm006}
	Let $R'$ be a positive integer prime to $p$ and suppose we have $v_{K_{R'}}(s_{R',i,j})=-qr_{i,j}+{R'}<-q\mu_{i,j}+pR'$. Then we have $v_{K_R}(s_{R,i,j})=-qr_{i,j}+R$ for any integer $R>R'$ congruent to $R'$ modulo $p$.
\end{lemma}
\begin{proof}

	Combining Lemma \ref{lm009} with $v_{K_{R'}}(s_{R',i,j})\equiv {R'}\pmod p$, we get $v_{K_{R'}}(s_{R',i,j}^{[1]})=-qr_{i,j}+R'$ and $v_{K_{R'}}(s_{R',i,j}^{[m]})>-qr_{i,j}+R'$ for all $m\geq 2$.

	Since $R\equiv R'\pmod p$, we have $\binom{-\frac{l}{R}}{m}\equiv \binom{-\frac{l}{R'}}{m}\pmod p$ for any integer $l$ and $1\leq m<p$. Therefore, by \eqref{eq061}, we have $\coef_{K_{R'},pl+mR'}(s_{R',i,j}^{[m]})=\coef_{K_{R},pl+mR}(s_{R,i,j}^{[m]})$ for any integer $l$ and $1\leq m<p$. This implies $v_{K_{R}}(s_{R,i,j}^{[m]})=v_{K_{R'}}(s_{R',i,j}^{[m]})+m(R-R')$ for $1\leq m<p$. Therefore we have
	\begin{equation}
		v_{K_{R}}(s_{R,i,j}^{[m]})=v_{K_{R'}}(s_{R',i,j}^{[m]})+m(R-R')>-qr_{i,j}+R'+m(R-R')>-qr_{i,j}+R
	\end{equation}
	for $2\leq m<p$ and
	\begin{equation}
		v_{K_{R}}(s_{R,i,j}^{[m]})\geq -q\mu_{i,j}+pR>-qr_{i,j}+R
	\end{equation}
	for $m\geq p$. Thus we have $v_{K_{R}}(s_{R,i,j})=v_{K_{R}}(s_{R,i,j}^{[1]})=-qr_{i,j}+R$.

\end{proof}

\begin{proposition}\label{pp002}
	Let $1\leq i<j\leq n$. Fix an integer $R>R_{i,j}$ prime to $p$, where $R_{i,j}$ is defined as \eqref{eq085}. Suppose we have $v_{K_R}(s_{R,l,m})<-q\mu_{l,m}+pR$ for $i\leq l<m\leq j$. Then we have $v_{K_R}(s_{R,i,j})=v_{K_{R}}(d_{R,i,j})=-qr_{i,j}+R$.
\end{proposition}
\begin{proof}
	We will prove the claim by induction on $j-i$.

	\begin{enumerate}[label=\textlangle \arabic*\textrangle]
		\item Case $j-i=1$

		      We have $s_{R,i,i+1}=\eta_R(a_{i,i+1})$. By \eqref{eq035} and Lemma \ref{lm014}, we have $v_{K_R}(s_{R,i,i+1})=-qm_{i,i+1}+R=-qr_{i,i+1}+R$.
		      By \eqref{eq011} and the definition of $D_R$, we have $d_{R,i,i+1}=-P(\eta_{R}(\theta_{i,i+1}))=\eta_{R}(a_{i,i+1})$. Hence we have
		      \begin{equation}
			      v_{K_{R}}(d_{R,i,i+1})=v_{K_R}(s_{R,i,i+1}).
		      \end{equation}
		\item Case $j-i>1$

		      Note that by Corollary \ref{cr002}, we have $R_{i,j}>R_{l,m}$ for $i\leq l<m\leq j$ such that $j-i>m-l$. Therefore by the induction hypothesis, we get $v_{K_R}(s_{R,l,m})=v_{K_R}(d_{R,l,m})=-qr_{l,m}+R$ for all $i\leq l<m\leq j$ except for $(l,m)=(i,j)$. Hence, by Lemma \ref{lm012}, we get $v_{K_R}(d_{R,i,j}-s_{R,i,j})>R-R_{i,j}>-qr_{i,j}+R$.
		      Thus it remains to show $v_{K_R}(s_{R,i,j})=-qr_{i,j}+R$.

		      \begin{enumerate}[label=\textlangle \arabic{enumi}-\arabic*\textrangle]
			      \item\label{ot031} Case $R_{i,j}< R<(q-1)r_{i,j}$

			      Since $v_{K_R}(d_{R,i,j}-s_{R,i,j})>R-R_{i,j}> 0$, by Lemma \ref{lm005} and Lemma \ref{lm009}, we have $v_{K_{R}}(s_{R,i,j})=-qr_{i,j}+R$.
			      \item Case $R\geq (q-1)r_{i,j}$

			            By Lemma \ref{lm008}, we have $R_{i,j}<(q-1)r_{i,j}-p$. We can take an integer $R'\leq R$ such that $R'\equiv R\pmod p$ and $R_{i,j}\leq R'<(q-1)r_{i,j}$.

			            Therefore by \ref{ot031} and Lemma \ref{lm006}, we get $v_{K_{R}}(s_{R,i,j})=-qr_{i,j}+R$.
		      \end{enumerate}

	\end{enumerate}
\end{proof}

\begin{proposition}\label{pp003}
	Let $1\leq i<j \leq n$ and $R>R_{i,j}$ be an integer prime to $p$. Suppose we have $v_{K_R}(s_{R,l,m})=v_{K_R}(d_{R,l,m})=-qr_{l,m}+R$ for $i\leq l<m\leq j$ except for $(l,m)=(i,j)$ and $v_{K_R}(s_{R,i,j})\leq -qr_{i,j}+R$. Let $\lambda$ be a partition for $(i,j)$ satisfying \eqref{eq072} in Lemma \ref{lm021}. Then we have $v_{K_R}(s_{R,i,j})\leq -q\mu_{i,j}+\lvert\lambda\rvert R$ and the equality holds if and only if $\lvert\lambda\rvert =1$.
\end{proposition}
\begin{proof}
	By \eqref{eq072}, we have
	\begin{equation}\label{eq073}
		-q\mu_{i,j}+\lvert\lambda\rvert R=\sum_{u=0}^{\lvert\lambda\rvert -1}v_{K_R}(s_{R,\lambda_{u},\lambda_{u+1}}).
	\end{equation}
	\begin{enumerate}[label=\textlangle \arabic*\textrangle]
		\item Case $\lvert\lambda\rvert =1$

		      In this case, \eqref{eq073} becomes $v_{K_R}(s_{R,i,j})=-q\mu_{i,j}+R$.

		\item Case $\lvert\lambda\rvert >1$

		      By assumption and \eqref{eq073}, we get
		      \begin{equation}
			      -q\mu_{i,j}+\lvert\lambda\rvert R=\sum_{u=0}^{\lvert\lambda\rvert -1}(-qr_{\lambda_{u},\lambda_{u+1}}+R).
		      \end{equation}
		      By Corollary \ref{cr002}, we have $R>R_{i,j}>qr_{\lambda_{u},\lambda_{u+1}}$ and $r_{i,j}>r_{\lambda_{u},\lambda_{u+1}}$ for all $0\leq u\leq \lvert\lambda\rvert -1$. Therefore, we have
		      \begin{equation}
			      -qr_{\lambda_{u},\lambda_{u+1}}+R> -qr_{\lambda_{u},\lambda_{u+1}}+R_{i,j}>0,
		      \end{equation}
		      and
		      \begin{equation}
			      -qr_{\lambda_{u},\lambda_{u+1}}+R>-qr_{i,j}+R.
		      \end{equation}
		      Therefore, we have
		      \begin{equation}
			      -q\mu_{i,j}+\lvert\lambda\rvert R> (\lvert\lambda\rvert -1)\cdot 0+1\cdot(-qr_{i,j}+R)=-qr_{i,j}+R\geq v_{K_R}(s_{R,i,j}).
		      \end{equation}
	\end{enumerate}
\end{proof}

\section{Case $j-i\leq p+1$}\label{sc006}

Let $K$ be a complete discrete valuation field of equal characteristic $p>0$ with algebraically closed residue field $k$ and $L$ be a totally wildly ramified finite Galois extension of $K$ such that $\Gal(L/K)\simeq UT_n(\mathbb{F}_p)$ for some $n\geq 2$.
Fix an integer $N>\log_{p}(n\left(p^{\frac{n(n-1)}{2}}+1\right)m_{A}+p)+\frac{n(n-1)}{2}$ and $q=p^N$, and define $K_R$ as in Definition \ref{df003} for any positive integer $R$ prime to $p$. Fix $A$ satisfying the conditions of \ref{lm011} and define $A_R$, $W_{R,i}$, $D_R$, $S_R$ as in Definition \ref{df005} and Definition \ref{df006}. Define $r_{i,j}$ as in Section $\ref{sc004}$, $m_{i,j}$ and $\mu_{i,j}$ as in Definition \ref{df008}, and $a_{i,j}^{[m]}$ and $s_{R,i,j}^{[m]}$ as in Definition \ref{df009} for $1\leq i<j\leq n$ and $m\in\mathbb{Z}$. Define $\coef_{K,l}:K\to k$ as in Definition \ref{df002} for any $l\in \mathbb{Z}$.

We will show in this section that the premise of Proposition \ref{pp002} holds for $1\leq i<j\leq n$ with $j-i\leq p+1$.

\begin{lemma}\label{lm003}
	Let $m> m'\geq 2$ be integers. Let $x_{i,j}\ (1\leq i\leq j\leq m)$ be integers satisfying $x_{i,j}+x_{j,l}=x_{i,l}$ for all $1\leq i\leq j\leq l\leq m$. Then there exists $1\leq i<j\leq m$ such that $x_{i,j}\equiv 0\pmod{m'}$.
\end{lemma}
\begin{proof}
	By the pigeonhole principle, there exists $1\leq i<j\leq m$ such that $x_{1,i}\equiv x_{1,j}\pmod{m'}$. Then we have $x_{i,j}=x_{1,j}-x_{1,i}\equiv 0\pmod{m'}$.
\end{proof}

\begin{lemma}\label{lm017}
	Let $1\leq i<j \leq n$ and $R>R_{i,j}$ be an integer prime to $p$. Suppose we have $j-i\leq p+1$ and $v_{K_R}(s_{R,l,m})=v_{K_R}(d_{R,l,m})=-qr_{l,m}+R<-q\mu_{l,m}+pR$ for $i\leq l<m\leq j$ except for $(l,m)=(i,j)$. Then we have $v_{K_R}(s_{R,i,j})\leq -q\mu_{i,j}+\min(p,j-i)R$, and the equality holds if and only if $j-i=1$.
\end{lemma}
\begin{proof}

	Let $\lambda\in\Lambda_{i,j}$ be one of the shortest partitions satisfying \eqref{eq072} in Lemma \ref{lm021}.
	\begin{enumerate}[label=\textlangle \arabic*\textrangle]
		\item\label{ot034} Case $\lvert\lambda\rvert \leq p$ and $R_{i,j}<R<(q-1)r_{i,j}$

		By Lemma \ref{lm012}, we get
		\begin{equation}
			v_{K_R}(d_{R,i,j}-s_{R,i,j})>R-R_{i,j}>0.
		\end{equation}
		Hence, by Lemma \ref{lm005}, we get $v_{K_R}(s_{R,i,j})\leq -qr_{i,j}+R$. Then the assertion follows from Proposition \ref{pp003}.
		\item Case $\lvert\lambda\rvert \leq p$ and $R\geq (q-1)r_{i,j}$

		      By Lemma \ref{lm008}, we have $R_{i,j}<(q-1)r_{i,j}-p$. We can take an integer $R'\leq R$ such that $R'\equiv R\pmod p$ and $R_{i,j}\leq R'<(q-1)r_{i,j}$.
		      By Proposition \ref{pp002} and \ref{ot034}, we get $v_{K_{R'}}(s_{R',i,j})=-qr_{i,j}+R'<-q\mu_{i,j}+pR'$. By Lemma \ref{lm006}, we have $v_{K_{R}}(s_{R,i,j})=-qr_{i,j}+R$. Hence, the assertion follows from Proposition \ref{pp003}.
		\item Case $\lvert\lambda\rvert =p+1$

		      We will prove by contradiction that this case does not happen.

		      By assumption, we have $\lvert\lambda\rvert =p+1=j-i$. By Lemma \ref{lm024}, we have
		      \begin{equation}
			      m_{l,m}=\sum_{l'=l}^{m-1}m_{l',l'+1}
		      \end{equation}
		      for all $i\leq l<m\leq j$.

		      Then by Lemma \ref{lm003}, there exists $i\leq l<m\leq j$ such that $p\mid m_{l,m}$. This contradicts with the definition of $m_{l,m}$.
	\end{enumerate}

\end{proof}

\begin{theorem}\label{tm001}
	Fix $A$ satisfying the conditions of Lemma \ref{lm011} and let $\Theta=(\theta_{i,j})_{i,j}$ be a solution of the equation $X^{(p)}A=X$. Fix an integer $N$ and $q=p^N$ as in Section \ref{sc004}, and define $K_R$ as in Definition \ref{df003}. Define $S_R$ as in Definition \ref{df006} and $r_{i,j}$ as the largest upper ramification jump of $L_{i,j}/K$, where $L_{i,j}=K(\theta_{l,m}\ (i\leq l<m\leq j))$ for $1\leq i<j\leq n$, as in Section \ref{sc004}.

	Let $1\leq i<j \leq n$ and $j-i\leq p+1$.
	Then we have
	$v_{K_R}(s_{R,i,j})=v_{K_R}(s_{R,i,j}^{[1]})=-qr_{i,j}+R$ for any integer $R>R_{i,j}$ prime to $p$.
\end{theorem}
\begin{proof}
	By using Lemma \ref{lm017} and Proposition \ref{pp002} inductively, we get $v_{K_R}(s_{R,i,j})=-qr_{i,j}+R<-q\mu_{i,j}+pR$. By the same argument as in the proof of Lemma \ref{lm009}, we get $v_{K_R}(s_{R,i,j})=v_{K_R}(s_{R,i,j}^{[1]})$.
\end{proof}

This implies that we can calculate $r_{i,j}$ for $j-i\leq p+1$ by calculating $s_{R,i,j}\in K_R$ for sufficiently large $R$.

\begin{corollary}\label{cr003}
	Let $1\leq i<j \leq n$ and $j-i\leq p+1$.
	Then we have $r_{i,j}\leq \mu_{i,j}$, and the equality holds if and only if the equality in \eqref{eq115} holds.
\end{corollary}
\begin{proof}
	This follows from Theorem \ref{tm001} and Lemma \ref{lm016}\ref{ot032}.
\end{proof}

\begin{remark}
Note that we have $\mathcal{U}_{L/K}\supset\{r_{i,j}\}_{1\leq i<j \leq n}$ and $\max\mathcal{U}_{L/K}=r_{1,n}$, but we do not always have $\mathcal{U}_{L/K}=\{r_{i,j}\}_{1\leq i<j \leq n}$. For example, let $p>2, n=3$ and suppose that $L_{1,2}/K$ is defined by $P(x)=t^2$ and $L_{2,3}/K$ is defined by $P(x)=t^2+t$. Then we have $r_{1,2}=r_{2,3}=2<r_{1,3}$. However, $L_{1,2}L_{2,3}$ includes the roots of $P(x)-t$, so $1\in\mathcal{U}_{L_{1,2}L_{2,3}/K}\subset \mathcal{U}_{L/K}$.

In general, we have $\mathcal{U}_{L/K}=\{r_{i,j}\}_{1\leq i<j \leq n}$ if $r_{i,j}$ for $1\leq i<j\leq n$ do not coincide with each other at all, since $\lvert\mathcal{U}_{L/K}\rvert \leq \frac{n(n-1)}{2}$.

On the other hand, we have $\lvert\{r_{i,j}\}\rvert\geq n-1$, since we have $r_{1,2}<r_{1,3}<\cdots<r_{1,n}$ by Corollary \ref{cr002}.

\end{remark}

\begin{example}\label{ex001}
	Consider the case $n=3$. Suppose $R$ is a sufficiently large positive integer prime to $p$.

	By definition, we have
	\begin{equation}
		s_{R,1,3}=a_{1,3}^{[1]}+a_{1,2}^{[1]}w_{R,N-1,2,3}^p-w_{R,N,1,2}^pa_{2,3}^{[1]}
	\end{equation}
	By the definition of $W_R$, we have
	\begin{align}
		 & w_{R,N,1,2}=\sum_{e=0}^{N-1}a_{R,1,2}^{p^e},   \\
		 & w_{R,N-1,2,3}=\sum_{e=0}^{N-2}a_{R,2,3}^{p^e}.
	\end{align}
	Therefore, we get
	\begin{equation}
		s_{R,1,3}^{[1]}=a_{1,3}^{[1]}-a_{R,1,2}^qa_{2,3}^{[1]}+\sum_{e=1}^{N-1}\left(a_{1,2}^{[1]}a_{R,2,3}^{p^e}-a_{R,1,2}^{p^e}a_{2,3}^{[1]}\right).
	\end{equation}
	In addition, we have
	\begin{align}
		 & v_{K_R}(a_{1,3}^{[1]})=-qm_{1,3}+R,                                                                                         &                    \\
		 & v_{K_R}(a_{R,1,2}^qa_{2,3}^{[1]})=-q(m_{1,2}+m_{2,3})+R,                                                                    &                    \\
		 & v_{K_R}\left(a_{1,2}^{[1]}a_{R,2,3}^{p^e}\right)=-q\left(m_{1,2}+\frac{m_{2,3}}{p^{N-e}}\right)+R&\geq -q(m_{1,2}+m_{2,3})+R\notag\\
		 & & (1\leq e\leq N-1), \\
		 & v_{K_R}\left(a_{R,1,2}^{p^e}a_{2,3}^{[1]}\right)=-q\left(\frac{m_{1,2}}{p^{N-e}}+m_{2,3}\right)+R&\geq -q(m_{1,2}+m_{2,3})+R \notag\\
		 && (1\leq e\leq N-1).
	\end{align}
	Since $A$ satisfies condition \ref{ot029} of Lemma \ref{lm011}, by Lemma \ref{lm016}\ref{ot043}, we have
	\begin{equation}
		v_{K_R}\left(a_{1,2}^{[1]}a_{R,2,3}^{p^e}-a_{R,1,2}^{p^e}a_{2,3}^{[1]}\right)\in \left(\frac{q}{p^e}\mathbb{Z}+R\right)-\left(\frac{q}{p^{e-1}}\mathbb{Z}+R\right)
	\end{equation}
	for $1\leq e\leq N-1$, and consequently,
	\begin{equation}
		v_{K_R}\left(\sum_{e=1}^{N-1}\left(a_{1,2}^{[1]}a_{R,2,3}^{p^e}-a_{R,1,2}^{p^e}a_{2,3}^{[1]}\right)\right)\notin q\mathbb{Z}+R.
		\stepcounter{equation}\tag{\theequation}
	\end{equation}
	By Lemma \ref{lm016}\ref{ot043}, we have
	\begin{equation}
		\coef_{K_{R},l}\left(a_{1,3}^{[1]}-a_{R,1,2}^qa_{2,3}^{[1]}\right)\neq 0 \Rightarrow l\in q\mathbb{Z}+R,
	\end{equation}
	and consequently,
	\begin{equation}
		v_{K_R}\left(a_{1,3}^{[1]}-a_{R,1,2}^qa_{2,3}^{[1]}\right)\in q\mathbb{Z}+R.
	\end{equation}
	Therefore, we have
	\begin{align*}\label{eq090}
	v_{K_R}(s_{R,1,3}^{[1]})&		=\min\left(v_{K_R}\left(a_{1,3}^{[1]}-a_{R,1,2}^qa_{2,3}^{[1]}\right),\vphantom{v_{K_R}\left(\sum_{e=1}^{N-1}\left(a_{1,2}^{[1]}a_{R,2,3}^{p^e}-a_{R,1,2}^{p^e}a_{2,3}^{[1]}\right)\right)}\right.\\
		&\left.\hphantom{=\min\left(\vphantom{v_{K_R}\left(\sum_{e=1}^{N-1}\left(a_{1,2}^{[1]}a_{R,2,3}^{p^e}-a_{R,1,2}^{p^e}a_{2,3}^{[1]}\right)\right)}\right.}v_{K_R}\left(\sum_{e=1}^{N-1}\left(a_{1,2}^{[1]}a_{R,2,3}^{p^e}-a_{R,1,2}^{p^e}a_{2,3}^{[1]}\right)\right)\right).
			      \stepcounter{equation}\tag{\theequation}
		      \end{align*}
	In \cite{Im}, it is shown that we may take $A$ such that at least one of the following conditions holds.
	\begin{enumerate}[label=(\alph*)]
		\item\label{ot041} $m_{1,2}\neq m_{2,3}$
		\item\label{ot042} $m_{1,2}=m_{2,3}$ and
		\begin{equation}
			\frac{\coef_{K,-m_{1,2}}(a_{1,2})}{\coef_{K,-m_{2,3}}(a_{2,3})}\notin \mathbb{F}_p
		\end{equation}
	\end{enumerate}
	If one of these conditions holds, we have
	\begin{equation}
		v_{K_R}\left(a_{1,2}^{[1]}a_{R,2,3}^{p^{N-1}}-a_{R,1,2}^{p^{N-1}}a_{2,3}^{[1]}\right)=-q\max\left(m_{1,2}+\frac{m_{2,3}}{p},\frac{m_{1,2}}{p}+m_{2,3}\right)+R.
	\end{equation}
	Since we have
	\begin{align*}
		v_{K_R}\left(a_{1,2}^{[1]}a_{R,2,3}^{p^e}-a_{R,1,2}^{p^e}a_{2,3}^{[1]}\right) & \geq -q\max\left(m_{1,2}+\frac{m_{2,3}}{p^{N-e}},\frac{m_{1,2}}{p^{N-e}}+m_{2,3}\right)+R \\
		                                                                              & >-q\max\left(m_{1,2}+\frac{m_{2,3}}{p},\frac{m_{1,2}}{p}+m_{2,3}\right)+R
		\stepcounter{equation}\tag{\theequation}
	\end{align*}
	for $1\leq e\leq N-2$, we get
	\begin{align*}
		&v_{K_R}\left(\sum_{e=1}^{N-1}\left(a_{1,2}^{[1]}a_{R,2,3}^{p^e}-a_{R,1,2}^{p^e}a_{2,3}^{[1]}\right)\right)\\
		&=-q\max\left(m_{1,2}+\frac{m_{2,3}}{p},\frac{m_{1,2}}{p}+m_{2,3}\right)+R.
			      \stepcounter{equation}\tag{\theequation}
		      \end{align*}
	Therefore, by \eqref{eq090} and Theorem \ref{tm001}, we get
	\begin{equation}\label{eq129}
		r_{1,3}=\max\left(-\frac{v_{K_R}\left(a_{1,3}^{[1]}-a_{R,1,2}^qa_{2,3}^{[1]}\right)-R}{q},m_{1,2}+\frac{m_{2,3}}{p},\frac{m_{1,2}}{p}+m_{2,3}\right)
	\end{equation}
	This result corresponds to the result given in \cite[Example 3.8.1]{Im}.
\end{example}

\begin{example}\label{ex002}
	Consider the case $n=4$. Suppose $R$ is a sufficiently large positive integer prime to $p$.

	Let
	\begin{align}
		a'_0&=    a_{1,4}^{[1]}-a_{R,1,2}^{q}a_{2,4}^{[1]}-\left(a_{R,1,3}^{q}-a_{R,1,2}^{q}a_{R,2,3}^{q}\right)a_{3,4}^{[1]}, \\
		a''_0&=   0                                                                                                            \\
		\mu'_0&=  \max\left(m_{1,4},m_{1,2}+m_{2,4},m_{1,3}+m_{3,4},m_{1,2}+m_{2,3}+m_{3,4}\right)=\mu_{1,4},
	\end{align}
	and
	\begin{align}
		a'_e&=    a_{1,3}^{[1]}a_{R,3,4}^{\frac{q}{p^e}}+a_{1,2}^{[1]}a_{R,2,4}^{\frac{q}{p^e}}-a_{R,1,3}^{\frac{q}{p^e}}a_{3,4}^{[1]}\notag\\
		&\hphantom{=\vphantom{a}}+a_{R,1,2}^{\frac{q}{p^e}}\left(a_{R,2,3}^{q}a_{3,4}^{[1]}-a_{2,4}^{[1]}\right)-a_{R,1,2}^{q}a_{2,3}^{[1]}a_{R,3,4}^{\frac{q}{p^e}}, \\
		a''_e&=   a_{R,1,2}^{\frac{q}{p^e}}\sum_{e'=1}^{e}a_{R,2,3}^{\frac{q}{p^{e'}}}a_{3,4}^{[1]}-\sum_{e'=1}^{e}a_{R,1,2}^{\frac{q}{p^{e'}}}a_{2,3}^{[1]}a_{R,3,4}^{\frac{q}{p^e}}-a_{R,1,2}^{\frac{q}{p^e}}a_{2,3}^{[1]}\sum_{e'=1}^{e-1}a_{R,3,4}^{\frac{q}{p^{e'}}}, \\
		\mu'_e&=  \max\left(m_{1,2}+\frac{m_{2,4}}{p^e},\frac{m_{1,2}}{p^e}+m_{2,4},m_{1,3}+\frac{m_{3,4}}{p^e},\frac{m_{1,3}}{p^e}+m_{3,4},\right.\notag		\\
		        & \left.\hphantom{=\max\left(\vphantom{m_{1,2}+\frac{m_{2,4}}{p^e},\frac{m_{1,2}}{p^e}+m_{2,4},m_{1,3}+\frac{m_{3,4}}{p^e},\frac{m_{1,3}}{p^e}+m_{3,4},}\right.} \frac{m_{1,2}}{p^e}+m_{2,3}+m_{3,4},m_{1,2}+m_{2,3}+\frac{m_{3,4}}{p^e}\right)
	\end{align}
	for $1\leq e\leq N-1$.
	Note that we have
	\begin{align}
		 & v_{K_R}(a'_e)\geq-q\mu'_e+R,\label{eq105} \\
		 & v_{K_R}(a''_e)>-q\mu'_e+R\label{eq111}
	\end{align}
	for $0\leq e\leq N-1$, and
	\begin{equation}\label{eq112}
		\mu'_e>\mu'_{e'}
	\end{equation}
	for $0\leq e<e'\leq N-1$.
	We also have
	\begin{align}
		 & \coef_{K_R,l}(a'_e)\neq 0\Rightarrow l\in \left(p^{N-e}\mathbb{Z}+R\right)-\left(p^{N-e+1}\mathbb{Z}+R\right),\label{eq113} \\
		 & \coef_{K_R,l}(a''_e)\neq 0\Rightarrow l\in p^{N-e}\mathbb{Z}+R\label{eq114}
	\end{align}
	for $0\leq e\leq N-1$, by Lemma \ref{lm016}\ref{ot043}.

	By Lemma \ref{lm019} and \eqref{eq061}, we get
	\begin{align*}
		s_{R,1,4}^{[1]}&=a_{1,4}^{[1]}+a_{1,3}^{[1]}w_{R,N-1,3,4}^p+a_{1,2}^{[1]}w_{R,N-1,2,4}^p\\
		&\hphantom{=\vphantom{a}}-\left(w_{R,N,1,3}^p-w_{R,N,1,2}^pw_{R,N,2,3}^p\right)a_{3,4}^{[1]} \\
		                 &\hphantom{=\vphantom{a}} -w_{R,N,1,2}^p\left(a_{2,4}^{[1]}+a_{2,3}^{[1]}w_{R,N-1,3,4}^p\right).\label{eq108}
		\stepcounter{equation}\tag{\theequation}
	\end{align*}
	By the definition of $W_R$ and \eqref{eq009}, we get
	\begin{align}
		  w_{R,e,i,i+1}^p&=\sum_{u=N-e}^{N-1}a_{R,i,i+1}^{\frac{q}{p^u}}                                                                                                  & (i=1,2,3,\ 1\leq e\leq N),\label{eq109} \\
		 w_{R,e,i,i+2}^p&=\sum_{u=N-e}^{N-1}a_{R,i,i+2}^{\frac{q}{p^u}}\notag\\
		 &\hphantom{=\vphantom{a}}+\sum_{u=N-e}^{N-1}\sum_{u'=u+1}^{N-1}a_{R,i,i+1}^{\frac{q}{p^u}}a_{R,i+1,i+2}^{\frac{q}{p^{u'}}} & (i=1,2,\ 1\leq e\leq N),
	\end{align}
	and consequently,
	\begin{equation}\label{eq110}
		w_{R,N,i,i+2}^p-w_{R,N,i,i+1}^pw_{R,N,i+1,i+2}^p=\sum_{u=0}^{N-1}a_{R,i,i+2}^{\frac{q}{p^u}}-\sum_{u=0}^{N-1}\sum_{u'=0}^{u-1}a_{R,i,i+1}^{\frac{q}{p^u}}a_{R,i+1,i+2}^{\frac{q}{p^{u'}}}
	\end{equation}
	for $i=1,2$.
	By substituting \eqref{eq109} and \eqref{eq110} into \eqref{eq108}, we get
	\begin{align*}
		s_{R,1,4}^{[1]}&=  a_{1,4}^{[1]}+a_{1,3}^{[1]}\sum_{e=1}^{N-1}a_{R,3,4}^{\frac{q}{p^e}}+a_{1,2}^{[1]}\sum_{e=1}^{N-1}a_{R,2,4}^{\frac{q}{p^e}}                                                                               \\
		                 &\hphantom{=\vphantom{a}} -\sum_{e=0}^{N-1}a_{R,1,3}^{\frac{q}{p^e}}a_{3,4}^{[1]}+\sum_{e=0}^{N-1}a_{R,1,2}^{\frac{q}{p^e}}\sum_{e'=0}^{e}a_{R,2,3}^{\frac{q}{p^{e'}}}a_{3,4}^{[1]}                                                 \\
		                 &\hphantom{=\vphantom{a}} -\sum_{e=0}^{N-1}a_{R,1,2}^{\frac{q}{p^e}}a_{2,4}^{[1]}-\sum_{e'=0}^{N-1}a_{R,1,2}^{\frac{q}{p^{e'}}}a_{2,3}^{[1]}\sum_{e=1}^{N-1}a_{R,3,4}^{\frac{q}{p^{e}}}                                             \\
		&=                 a_{1,4}^{[1]}-a_{R,1,2}^{q}a_{2,4}^{[1]}-\left(a_{R,1,3}^{q}-a_{R,1,2}^{q}a_{R,2,3}^{q}\right)a_{3,4}^{[1]}                                                                                               \\
		                 &\hphantom{=\vphantom{a}} +\sum_{e=1}^{N-1}\left(a_{1,3}^{[1]}a_{R,3,4}^{\frac{q}{p^e}}+a_{1,2}^{[1]}a_{R,2,4}^{\frac{q}{p^e}}-a_{R,1,3}^{\frac{q}{p^e}}a_{3,4}^{[1]}-a_{R,1,2}^{\frac{q}{p^e}}a_{2,4}^{[1]}\right)                 \\
		                 & \hphantom{=\vphantom{a}}+\sum_{e=1}^{N-1}a_{R,1,2}^{\frac{q}{p^e}}\sum_{e'=0}^{e}a_{R,2,3}^{\frac{q}{p^{e'}}}a_{3,4}^{[1]}-\sum_{e'=0}^{N-1}\sum_{e=1}^{N-1}a_{R,1,2}^{\frac{q}{p^{e'}}}a_{2,3}^{[1]}a_{R,3,4}^{\frac{q}{p^{e}}}.
		\stepcounter{equation}\tag{\theequation}
	\end{align*}
	By separating the terms and changing the order of the summations, we can rewrite the last two terms as follows:
	\begin{equation}
		\sum_{e=1}^{N-1}a_{R,1,2}^{\frac{q}{p^e}}\sum_{e'=0}^{e}a_{R,2,3}^{\frac{q}{p^{e'}}}a_{3,4}^{[1]}-\sum_{e'=0}^{N-1}\sum_{e=1}^{N-1}a_{R,1,2}^{\frac{q}{p^{e'}}}a_{2,3}^{[1]}a_{R,3,4}^{\frac{q}{p^{e}}}=\sum_{e=1}^{N-1}a_{R,1,2}^{\frac{q}{p^e}}a_{R,2,3}^{q}a_{3,4}^{[1]}+\sum_{e=1}^{N-1}a''_{e}.
	\end{equation}
	Therefore, we get
	\begin{equation}
		s_{R,1,4}^{[1]}=\sum_{e=0}^{N-1}a'_e+\sum_{e=0}^{N-1}a''_e.
	\end{equation}
	By \eqref{eq105}, \eqref{eq111}, and \eqref{eq112}, we obtain
	\begin{equation}\label{eq128}
		s_{R,1,4}^{[1]}\equiv\sum_{e=0}^{e'}a'_e+\sum_{e=0}^{e'-1}a''_e\pmod{\mathfrak{m}_{K_R}^{-q\mu'_{e'}+R+1}}
	\end{equation}
	for $0\leq e'\leq N-1$.

	Suppose the equality holds in \eqref{eq105} for some $e\in \mathbb{Z}_{\geq 1}$ and $e_0$ be one of such $e$. By \eqref{eq113} and \eqref{eq114}, we have
	\begin{equation}
		v_{K_R}\left(\sum_{e=0}^{e_0-1}a'_e+\sum_{e=0}^{e_0-1}a''_e\right)\in p^{N-e_0+1}\mathbb{Z}+R,
	\end{equation}
	and
	\begin{equation}
		v_{K_R}(a'_{e_0})=-q\mu'_{e_0}+R\in \left(p^{N-e_0}\mathbb{Z}+R\right)-\left(p^{N-e_0+1}\mathbb{Z}+R\right).
	\end{equation}
	Therefore, combining these with \eqref{eq128}, we deduce that
	\begin{equation}
		v_{K_R}(s_{R,1,4}^{[1]})=\min\left(v_{K_R}\left(\sum_{e=0}^{e_0-1}a'_e+\sum_{e=0}^{e_0-1}a''_e\right),-q\mu'_{e_0}+R\right).
	\end{equation}
	By Theorem \ref{tm001}, we deduce that
	\begin{equation}\label{eq130}
		r_{1,4}=\max\left(-\frac{v_{K_R}\left(\sum_{e=0}^{e_0-1}a'_e+\sum_{e=0}^{e_0-1}a''_e\right)-R}{q},\mu'_{e_0}\right).
	\end{equation}

	Note that the result \eqref{eq129} in Example \ref{ex001} has a similar form to \eqref{eq130}. To clarify the similarity, we make a similar argument in the case $n=3$. Then the condition that either \ref{ot041} or \ref{ot042} holds in Example \ref{ex001} corresponds to the condition that the equality in \eqref{eq105} holds for $e=1$. This implies that we have $e_0=1$ in the case $n=3$. Then we can deduce the equation \eqref{eq129} in the same way as \eqref{eq130}.
\end{example}





\subsection*{Acknowledgments}
The author wishes to express her gratitude to her supervisor Professor Takeshi Saito, who provided helpful suggestions and comments throughout the research process and the preparation of this paper.


\end{document}